\documentclass[11 pt, a4paper]{amsart}
\usepackage{amssymb,amsmath,epsfig,mathrsfs, enumerate}
\usepackage{graphicx}
\usepackage{fancyhdr}
\fancyheadoffset{0cm}
\usepackage[margin=3cm]{geometry}
\usepackage{hyperref}
\hypersetup{
 colorlinks   = true,
 urlcolor     = blue,
 linkcolor    = blue,
 citecolor   = red ,
 bookmarksopen=true
}

\newtheorem{theorem}{Theorem}[section]
\theoremstyle{plain}

\newtheorem{corollary}[theorem]{Corollary}

\newtheorem{definition}[theorem]{Definition}

\newtheorem{lemma}[theorem]{Lemma}

\newtheorem{proposition}[theorem]{Proposition}
\newtheorem{remark}[theorem]{Remark}

\numberwithin{equation}{section}
\newcommand{\R}{\mathbb{R}}
\newcommand{\pa}{\partial}
\newcommand{\Om}{\Omega}

\newcommand{\Hn}{\mathbb{H}^{n}}

\newcommand{\g}{\mathfrak{g}}
\newcommand{\Tau}{\Gamma}
\newcommand{\G}{\mathbb G}

\newcommand{\ve}{\varepsilon}
\newcommand{\cH}{\mathscr{H}}
\newcommand{\vf}{\varphi}
\newcommand{\Ri}{\mathscr R}
\newcommand{\V}{\mathscr W}

\title[]{Borderline gradient estimates at the  boundary in Carnot groups}

\author{Ramesh Manna}
\address[Ramesh Manna]{TIFR Centre for Applicable Mathematics, Sharada Nagar, Chikkabommasandra, Bangalore-560065, India} \email{ramesh@tifrbng.res.in}
\author{Ram Baran Verma}
\address[Ram Baran Verma]{TIFR Centre for Applicable Mathematics, Sharada Nagar, Chikkabommasandra, Bangalore-560065, India} \email{rambaran@tifrbng.res.in}
\subjclass[2010]{Primary 35J25; 35J60; Secondary 35D40}
\date{}
\keywords{Second-order subelliptic equations, borderline gradient continuity, Dini-domains, Carnot group}
\begin{document}
\begin{abstract}
In this article, we prove the continuity of the horizontal gradient near a  $C^{1,\text{Dini}}$ non-characteristic portion of the boundary for solutions to $\Tau^{0, \text{Dini}}$ perturbations of horizontal Laplaceans  as in \eqref{dp1} below where the scalar term is in scaling critical Lorentz space $L(Q,1)$ with $Q$ being the homogeneous dimension of the group. This result  can be thought of both  as a sharpening of the $\Tau^{1, \alpha}$ boundary  regularity result  in  \cite{BG} as well as a subelliptic analogue of the main result in  \cite{KA} restricted to  linear equations.
\end{abstract}

\maketitle

\tableofcontents

\section{Introduction}
In this article we consider the following boundary value problem
\begin{equation}\label{dp1}
\left\{
\begin{aligned}{}
\sum_{i,j=1}^m X_i^\star (a_{ij}X_j u)&=\sum_{i=1}^m X_i^\star f_i + g~~~~~\text{in}\ \Omega \subset \G,\\
u&= h\ \ \ \text{on}\ \partial\Omega
\end{aligned}
\right.
\end{equation}
where $[a_{ij}]$ is an $m\times m$ real symmetric matrix satisfying the following ellipticity condition
\begin{equation}\label{ea01}
\lambda \mathbb{I}_m \leq \mathbb{A}(p) \leq \lambda^{-1} \mathbb{I}_m,~~ p\in \G
\end{equation}
for some $\lambda>0.$ In \eqref{ea01}, $\mathbb{I}_m$ stands for $m\times m$ identity matrix and $\G$ is a Carnot group of step $k$ (see, Definition \ref{D:CG}). The central position of such Lie groups in the analysis of the hypoelliptic operators introduced by H\"ormander in his famous paper \cite{H} was established in the 1976 work of Rothschild and Stein on the so-called \emph{lifting theorem}, see \cite{RS,Snice}. Here, our aim is to obtain the pointwise gradient estimate for  weak solutions to  \eqref{dp1}  upto the non-characteristic portion of the boundary under minimal regularity assumptions on $[a_{ij}],$ $f_i,$ $g,h$ and the boundary $\partial\Omega.$

The fundamental role of such borderline  regularity results in the context of elliptic and parabolic equations is well known.  By using the well established theory of singular integral in the setting of Heisenberg group, interior Schauder estimate has been studied by many authors in \cite{13,14,44,47, GL, LU, XZ} and the reference therein. They play an important role in the analysis of nonlinear PDE's.

In 1981, D. Jerison in his famous works \cite{Je1,Je2} addressed the question of Schauder estimate at the boundary for the horizontal Laplacian in the Heisenberg group $\Hn.$  Jerison divided his analysis in two parts, according to whether or not the relevant portion of the boundary contains so-called \emph{characteristic points} (see Definition \ref{D:char}). At such points the vector fields that form the relevant differential operator become tangent to the boundary  and thus one should expect a sudden loss of differentiability,  somewhat akin to what happens in the classical setting with oblique derivative problems. In fact, Jerison proved that there exist no Schauder boundary estimates at characteristic points! He did so by constructing a domain in $\Hn$ with real-analytic boundary that support solutions of the horizontal Laplacian $\Delta_\cH u = 0$ which vanish near a characteristic boundary point, and which near such point possess no better regularity than H\"older's.  On the other hand he established  Schauder estimates at the non characteristic portion of the boundary.

Very recently in \cite{BCC},  by suitably adapting the Levi's method of parametrix,  Baldi, G. Citti and G. Cupini  established $\Tau^{2, \alpha}$ type  Schauder estimate for non-divergence form operators upto the non-characteristic portion of  a $C^{\infty}$  boundary in more general Carnot groups. Subsequently   in \cite{BG}, by employing an alternate approach based on geometric compactness arguments,  the authors  showed the  validity of $\Tau^{1, \alpha}$  boundary Schauder estimate for  divergence form operators as in  \eqref{dp1} when  boundary is $C^{1, \alpha}$ regular  and  when $a_{ij},f_{i}\in\Gamma^{0,\alpha}, h\in\Gamma^{1,\alpha}, g \in L^{\infty}$.  We note that such compactness arguments has its roots in the fundamental works of Caffarelli as in \cite{Ca} and is  independent of the method of parametrix.  In this article,  we consider  a similar framework  as in \cite{BG} and prove the horizontal continuity of the gradient under weaker assumptions on the coefficients and the domain and when the scalar term $g$ belongs to the scaling critical Lorentz space $L(Q,1)$ with $Q$ being the homogeneous dimension of the Carnot group $\mathbb{G}$. For the precise notion of the function space $L(Q,1)$, we refer the reader  to definition \ref{lorentz1}.

Finally in order to put our results in the right perspective, we note that in 1981,  E. Stein in his visionary work \cite{EM} showed the following "limiting" case of Sobolev embedding theorem.
\begin{theorem}\label{stein}
	Let $L(n,1)$ denote the standard Lorentz space, then the following implication holds:
	\[\nabla v \in L(n,1) \ \implies \ v\  \text{is continuous}.\]
\end{theorem}
The Lorentz space $L(n,1)$ appearing in Theorem \ref{stein}  consists of those measurable functions $g$ satisfying the condition
\[
\int_{0}^{\infty} |\{x: g(x) > t\}|^{1/n} dt < \infty.
\]
Theorem \ref{stein} can be regarded as the limiting case of Sobolev-Morrey embedding that asserts
\[
\nabla v \in L^{n+\ve} \implies v \in C^{0, \frac{\ve}{n+\ve}}.
\]
Note that indeed $L^{n+\ve} \subset L(n, 1) \subset L^{n}$ for any $\ve>0$ with all the inclusions being strict.  Now Theorem \ref{stein} coupled with the standard Calderon-Zygmund theory  has the following interesting consequence.
\begin{theorem}\label{st}
	$\Delta u \in L(n,1) \implies \nabla u$ is continuous.
\end{theorem}
The analogue of Theorem \ref{st}  for general nonlinear and possibly degenerate elliptic and parabolic equations has become accessible not so long ago through a rather sophisticated and powerful nonlinear potential theory (see for instance \cite{M1,M2, M3, M4, M5, M6} and the references therein).  The first breakthrough in this direction came up in the work  of Kuusi and Mingione in \cite{M4} where they showed that the analogue of Theorem \ref{st}  holds for operators modelled after the $p$-Laplacian. Such a result  was subsequently generalized to $p$-Laplacian type systems by the same authors in  \cite{M5}.

Since then, there has been several generalizations of Theorem \ref{st} to operators with other kinds of nonlinearities and in context of  fully nonlinear elliptic equations. For instance, the gradient potential estimate for fully nonlinear elliptic equations has been established by Daskalopoulos-Kuusi-Mingione, see \cite[Theorem 1.2]{M16}.  We  also refer to  \cite{KA} for the boundary analogue of the regularity result in  \cite{M16}  and also to the more recent work \cite{BM} for similar borderline regularity results in the context of normalized $p$-Laplacian.  We note that the main idea in order to establish such end point gradient continuity estimates is to employ the modified Riesz potential defined as follows.
\begin{eqnarray} \label{ig}
\tilde{I}_q^g(p,R)=\int^{R}_{0}\Big(\frac{1}{|\Omega\cap B(p,\tau)|}\int_{\Omega\cap B(p,\tau)}|g(x)|^{q}dx\Big)^{\frac{1}{q}}d\tau,
\end{eqnarray}
where $B(p,\tau)$ is defined  as in \eqref{pseudo} below.  In fact, one estimate the $L^{\infty}$ norm of the gradient as well as a certain  moduli of continuity estimate in terms of such  modified Riesz potential. Then the continuity of the gradient follows from the fact that
\begin{equation}\label{potential}
\tilde{I}_q^g(p,R)\rightarrow0~\text{as}~R\rightarrow0
\end{equation}
provided $g\in L(Q,1)$ and $q<Q,$ for the details, see \cite[Theorem 1.3]{M16}. We will follow a  similar approach to prove our main result  Theorem \ref{main}. $\alpha$-decreasing (see Definition 2.11) property
of the modulus of continuity will play an important role in our arguments.

Taking these considerations into account,  we initiate the study of the regularity property of the solution of \eqref{dp1}. In order to state the main theorem, we  introduce a few relevant notations. Given an open set $\Om\subset \G,$ $p_0\in \pa \Om$ and $\tau>0,$ we set
\begin{equation}\label{VS}
\V_\tau = \Om\cap B(p_0,\tau),\ \ \ \ \ \ \mathscr S_\tau = \pa\Om\cap B(p_0,\tau).
\end{equation}
We now state our main theorem.
\begin{theorem}\label{main}
	Let $\Om \subset \G$ be of class $C^{1,\text{Dini}}$ and  $p_0 \in \pa \Om$ be such that for some $\tau>0$ we have that  the set $\mathscr S_\tau$ consists only of  non-characteristic points . Let $u\in \mathscr{L}^{1,2}_{loc}(\V_\tau) \cap C(\overline{\V_\tau})$ be a weak solution to \eqref{dp1}, with $a_{ij}, f_i, g$ and $h$ satisfying the following hypothesis:
	\begin{equation}\label{assump12}
	a_{ij} \in \Tau^{0, \text{Dini}}(\overline{\V_\tau}),\ \ f_i \in \Tau^{0,\text{Dini}}(\overline{\V_\tau}),\ \ g \in L(Q,1),\ \ h \in \Tau^{1,\text{Dini}}(\overline{\V_\tau}).
	\end{equation}
	Moreover, we also assume that the uniform ellipticity condition as in \eqref{ea01} holds. Then $\nabla_\cH u$ is continuous in  $\overline{\V_{\tau/2}}$, and moreover for any $p,q\in\overline{\V_{\tau/2}},$ we have the following estimate:
	\begin{equation}\label{ap}
	|\nabla_\cH u(p)-\nabla_\cH u(q)| \leq C_{1}W(C_{0}d(p,q)),
	\end{equation}
	where $d(\cdot,\cdot)$ is defined by \eqref{rho0}, $\nabla_\cH u$ stands for the horizontal gradient $u$ and $W$ is a modulus of continuity given by \eqref{modu}.
\end{theorem}
Our proof will consist of five main steps.
Though the idea of proof of our main Theorem \ref{main} is motivated by the work of Agnid et. al in \cite{BG}, but due to the lack of the enough regularity on the data and boundary we obtain abstract modulus of continuity of the horizontal gradient instead of the H\"{o}lder modulus of continuity. The presence of the abstract modulus of continuity poses additional difficulty in the proof. For instance one can see steps 3, 4 and 5 in the proof of Theorem \ref{main}. In the step 3 we prove the existence of Taylor polynomial at non-characteristic portion of the boundary, which follows from the mathematical induction in combination with compactness Lemma \ref{compactness}. In order to apply the compactness lemma we define a new rescaled function by \eqref{rescaled} which contains the modulus of continuity $\omega.$ So in order to satisfy all the assumptions in the compactness lemma we need many properties of the modulus of continuity, which is given in step 2 of the proof of Theorem \ref{main}. Similarly, in the proof of continuity of the horizontal gradient on the non-characteristic portion of the boundary (see step 4 in the proof of Theorem \ref{main}) and up to the boundary (step 5 in the proof of Theorem \ref{main}) we need a suitable scaling invariant version of the interior estimate, see Corollary \ref{intr1}. This estimate is a suitable adaptation of Corollary 3.2 in \cite{BG} in our set up. In step 5, we patch up the interior and boundary estimate to get the continuity of the horizontal gradient up to the boundary. In the process of patching, we crucially use $\alpha$-decreasing property of the modulus of continuity.

The article is organized as follows. Section \ref{sec2} consists of some basic definitions concerning the Carnot group. We also collect some known regularity results that will be used in the proof of  Theorem \ref{main}. Section \ref{sec3} is devoted to the proof of our main result Theorem \ref{main}.

\section{Basic definitions and results} \label{sec2}
Before we proceed with the proof of our main Theorem, we need to state some of the basic definitions concerning the Carnot group, modulus of continuity of functions etc. and some of its properties that will be used throughout the article. In the last part of this section, some known regularity results also has been presented which will be needed in the proof of Theorem \ref{main}. Most of the definitions related to the Carnot group, we refer \cite{BG} for the details.
Let us start by defining the Carnot group.
\begin{definition}\label{D:CG}
Given $k\in \mathbb N$, a Carnot group of step $k$ is a simply-connected real Lie group $(\G, \circ)$ whose Lie algebra $\g$ is stratified and $k$-nilpotent. This means that there exist vector spaces $\g_1,...,\g_k$ such that
\begin{itemize}
\item[(1)] $\g=\g_1\oplus \dots\oplus\g_k$;
\item[(2)] $[\g_1,\g_j] = \g_{j+1}$, $j=1,...,k-1,\ \ \ [\g_1,\g_k] = \{0\}$.
\end{itemize}
\end{definition}
We let $m_j = \dim \g_j$, $j = 1, \dots, k,$ and denote by
$N = m_1 + \ldots + m_k$
the topological dimension of $\G$. For simplicity in the notation from here onwards we will write $m$ for $m_{1}.$ Since $\G$ is simply-connected, the exponential mapping $\exp: \g \to \G$ is a global analytic diffeomorphism onto, see for instance \cite{V,CGr}. We will use this global chart to identify the point $p = \exp \xi\in \G$ with its logarithmic preimage $\xi\in \g$.

Let us introduce analytic maps $\xi_i : \G \rightarrow \g_j$, $j = 1,\ldots,k$, by $p = \exp\left(\xi_1(p)+\ldots+ \xi_k(p)\right)$. For $p \in \G$, the projection of the \emph{logarithmic coordinates} of $p$ onto the layer $\g_j$, $j = 1,\ldots, k,$ are defined by
\begin{equation}
\label{coor}
x_{j,s}(p) = \langle \xi_j(p),e_{j,s}\rangle, \quad s = 1,\ldots,m_j,
\end{equation}
where $(x_1(p),...,x_m(p)) = (x_{1,1}(p),...,x_{1,m}(p))$ are the horizontal coordinates of $p$ and the sets $\{e_{j,1}, \dots, e_{j,m_j}\}$, $j = 1,\dots, k,$ are a fixed orthonormal basis of the $j$-th layer $\g_j$ of the Lie algebra $\g$.
Sometimes, we will  omit the dependence in $p$, and identify $p$ with its logarithmic coordinates
\begin{equation}\label{pvar}
p \cong (x_1,...,x_m, x_{2,1},...,x_{2,m_2},....,x_{k,1},...,x_{k,m_k}).
\end{equation}
For simplify the notation  let
\begin{equation}\label{xis}
\xi_1 = (x_1,...,x_m),\ \xi_2 = (x_{2,1},...,x_{2,m_2}), . . . , \xi_k = (x_{k,1},...,x_{k,m_k}).
\end{equation}
Furthermore, we write
$x = x(p) \cong \xi_1 = (x_1,...,x_m)$, and $y = y(p)$ the $(N - m)-$dimensional vector
\[
y \cong (\xi_2,...,\xi_k) = (x_{2,1},...,x_{2,m_2},....,x_{k,1},...,x_{k,m_k}).
\]
In this case, we will write $z = (x,y)$, see \cite{FS}. For every $j=1,...,k$ we also use the following multi-index notation
$\alpha_j = (\alpha_{j,1},...,\alpha_{j,m_j}) \in (\mathbb N\cup \{0\})^{m_j}.$

In this article, we assume that $\{e_1,...,e_m\}$ is  an orthonormal basis  of $\g_1$, and that $\{X_1,...,X_m\}$ are left-invariant vector fields on $\G$.
Note that, the vector fields $\{X_1,\ldots,X_m\}$ form a basis for the so-called  horizontal sub-bundle $\cH$ of the tangent bundle $T\G$.
Given a point $p\in \G$, the fiber of $\cH$ at $p$ is given by
\begin{equation}\label{fiber}
\cH_p = d\mathcal{L}_p(\g_1).
\end{equation}
\begin{definition}[Horizontal Laplacean]\label{D:sl}
The \emph{horizontal Laplacean} associated with an orthonormal basis $\{e_1,...,e_m\}$ of the horizontal layer $\g_1$ is the left-invariant second-order partial differential operator in $\G$ defined by
\begin{equation}\label{subl}
\Delta_\cH = -\sum^m_{j=1}X^{\star}_{j}X_j=\sum^m_{j=1}X^2_j,
\end{equation}
where $\{X_1,...,X_m\}$ are left-invariant vector fields on $\G$ and the formal adjoint of $X_j$ in $L^2(\G)$ is given by $X_j^{\star} = -X_j$.
\end{definition}

\subsection{Gauge pseudo-distance}\label{SS:intrinsic}
In a Carnot group there exists a left-invariant distance $d_C(p,p_0)$ associated with the horizontal subbundle $\cH$, see for instance \cite{Be, NSW} and Chapter 4 in \cite{Gems}. A piecewise $\mathbb{C}^1$ curve $\alpha:[0,T]\to \G$ is called \emph{horizontal} if there exist piecewise continuous functions $b_i:[0,T]\to \G$ with $\sum_{i=1}^m |b_i| \le 1$ such that
\[
\alpha'(t)=\sum_{i=1}^m b_i(t)X_i(\alpha(t)).
\]
We define the \emph{horizontal length} of $\alpha$ as $\ell_\cH(\alpha) = T$ and the metric
\[
d_C(p,p_0)=\underset{\alpha\in \Gamma(p,p_0)}{\inf} \ell_\cH(\alpha),~p, p_0\in \G
\]
where $\Gamma(p,p_0)$ is the collection of all horizontal curves $\alpha:[0,T]\to \G$ such that $\alpha(0)=p$ and $\alpha(T)=p_0$.
The metric $d_C(p,p_0)$ is called the \emph{Carnot-Carath\'eodory  distance}. We can always extend this metric to a full Riemannian metric in $\R^N$ so that its volume element is the Lebesgue measure $\mathcal{L}.$ By Chow's theorem \cite{BR}, any two points can be connected by a horizontal curve, which makes $d_C$ a metric on $\R^N.$

The Carnot-Carath\'eodory metric $d_C(p,p')$ is equivalent to a more explicitly
defined pseudo-distance function, called the  \emph{gauge pseudo-distance}, defined as follows.
Let $||\cdot||$ denote the Euclidean distance to the origin in $\g$. For $\xi = \xi_1 +\cdots+\xi_k \in \g$, $\xi_j \in \g_j$, $j = 1,\ldots,k$, we define
\begin{equation}\label{homd}
|\xi|_{\g} = \left(\sum^k_{j=1} ||\xi_j||^{\frac{2k!}{j}}\right)^{2k!}, \ \ \ \ \quad |p|_{\G}=|\exp^{-1} p|_{\g}\ \quad p \in \G.
\end{equation}
The function $p\to |p|_\G$ is called the non-isotropic group gauge and satisfies for any $\lambda>0$
\begin{equation}\label{gau}
|\delta_\lambda(p)| = \lambda |p|,
\end{equation}
where dilations $\{\delta_\lambda\}_{\lambda>0}$ are group automorphisms (see \cite{BG}) and $|p|=|p|_\G$.
The \emph{gauge pseudo-distance} in $\G$ is defined by
\begin{equation}\label{rho0}
d(p,p_0) = |p^{-1} \circ p_0|.
\end{equation}

Now we define the metric and the gauge pseudo ball centered at $p$ with radius $R$
\begin{equation}\label{pseudo}
B_C(p,R) = \{p_0 \in\G\ |\ d_C(p_0,p)<R\},\ \ \ \ \ \ B(p,R)=\{p_0 \in\G\ |\ d(p_0,p)<R\},
\end{equation}
respectively. When the center is the group identity $e$, we will write $B_C(R)$ and $B(R)$ instead of $B_C(e,R)$ and $B(e,R)$.
Now, we denote $|E| = \int_E dp$ the Haar measure of a set $E\subset \G$.
Note that $\omega_C = \omega_C(\G) = |B_C(1)|>0$ and $\omega = \omega(\G) = |B(1)|>0$, and hence for every $p \in \G$ and $R > 0$,
\begin{equation}\label{volb}
|B_C(p,R)| =\omega_C R^Q,\ \ \ \ \ \ \ \ \ \ \ |B(p,R)| =\omega R^Q.
\end{equation}
\begin{lemma}[\cite{NSW}]\label{L:dis}
	For every connected  $\Omega \subset\subset \G$ there exist $C,\varepsilon
	>0$ such that
	\begin{equation}\label{xy}
	C d_\Ri(p,p_0) \leq d_C(p,p_0)\leq C^{-1} d_\Ri(p,p_0)^\varepsilon,
	\end{equation}
	where $d_\Ri(x,y)$ is the left-invariant Riemannian distance in $\G$ and $p, p_0 \in \Omega$.
\end{lemma}
\subsection{The Folland-Stein H\"older classes}
Now, we recall the intrinsic H\"older classes $\Tau^{\kappa, \alpha}$ introduced by Folland and Stein in \cite{F} and especially \cite{FS} see also Chapter 20 in \cite{BLU}.

\begin{definition}\label{hf}
Let $0< \alpha \leq 1$. Given an open set $\Om \subset \G$ we say that $u:\Om\to \R$ belongs to $\Tau^{0, \alpha}(\Om)$ if there exists a positive constant $M$ such that for every $p, p_0 \in \Om$,
\[
|u(p) - u(p_0)| \leq M\ d(p,p_0)^{\alpha}.
\]
We define the semi-norm
\begin{equation}\label{semi}
[u]_{\Tau^{0,\alpha}(\Om)}= \underset{\underset{p \neq p_0}{p, p_0\in \Om}}{\sup} \frac{|u(p)-u(p_0)|}{d(p,p_0)^{\alpha}}.
\end{equation}
Given $\kappa\in \mathbb N$, the spaces $\Tau^{\kappa, \alpha}(\Om)$ are defined inductively: we say that $u \in \Tau^{\kappa, \alpha}(\Om)$ if $X_i u \in \Tau^{\kappa-1,\alpha}(\Om)$ for every $i=1, .., m$.
\end{definition}
Note that for any $\lambda>0,~
[\delta_\lambda u]_{\Tau^{0,\alpha}(\delta_{\lambda^{-1}}(\Om))} = \lambda^\alpha [u]_{\Tau^{0,\alpha}(\Om)},$
where dilations $\{\delta_\lambda\}_{\lambda>0}$ are group automorphisms, see for more details in \cite{BG}.
\begin{definition} [Sobolev space]
 For an  open set $\Om\subset \G$  we denote by $\mathscr{L}^{1,p}(\Om)$, where $1\leq p\leq\infty$,  the Sobolev space $\{f \in L^p(\Om)\ | \ X_jf \in L^p(\Om), j = 1,\ldots,m\}$ endowed with its natural norm
\[
\|f\|_{\mathscr{L}^{1,p}(\Om)} = \|f\|_{L^{p}(\Om)} +   \sum_{j=1}^m \|X_jf\|_{L^{p}(\Om)}.
\]
\end{definition}
\noindent
The local space $\mathscr{L}_{loc}^{1,p}(\Om)$ has the usual meaning. We also denote by $\mathscr{L}^{1,p}_0(\Om) = \overline{C^\infty_0(\Om)}^{||\cdot||_{\mathscr{L}^{1,p}(\Om)}}$. Let $\lambda$ denotes the distribution function of $f$ defined on $\R^N,$ then the non-increasing rearrangement $f^{\star}$ is defined for $t>0$ by letting
$$f^{\star}(t)=\inf \{s>0: \lambda(s)\leq t\}.$$

\begin{definition}[Lorentz spaces \cite{CD}] \label{lorentz1}
	Let $Q$  be  strictly positive numbers such that $Q>1.$ The Lorentz space $L(Q,1)(\R^N)$ is defined as the set of real valued, measurable functions $f$, defined on $\R^N,$ such that:
	$$\|f\|_{L(Q,1)(\R^N)}=\int_0^{\infty}(f^{\star}(t) \, t^{\frac{1}{Q}}) \frac{dt}{t} < \infty.$$
\end{definition}
Note that, Carnot group $\mathbb{G}$ endowed with the Carnot gauge $\|x\|_C=d_C(x,0)$ or with a smooth gauge $x \to |x|_{\g}$ together with the Lebesgue measure $\mathcal{L}$ forms a real variable rearrangement structure. For more details one refer to \cite[Theorem 3.1]{MVS}, see also \cite{aba}.

\subsection{The characteristic set}\label{SS:char}
We start with an open set $\Om \subset \G$ which belongs to a class $C^1$ that is, for every $p_0\in \pa\Om$ there exist a neighborhood $U_{p_0}$ of $p_0$, and a function $\vf_{p_0} \in C^1(U_{p_0})$, with $|\nabla \vf_{p_0}| \geq \alpha > 0$ in $U_{p_0}$, such that
\begin{equation}\label{ome}
\Om \cap U_{p_0} = \{p\in U_{p_0} \mid \vf_{p_0}(p)<0\}, \quad\pa\Om\cap U_{p_0} =\{p\in U_{p_0} \mid \vf_{p_0}(p)=0\}.
\end{equation}
At every point $p\in \partial \Om \cap U_{p_0}$ the outer unit normal is given by
\[
\nu(p) = \frac{\nabla \vf_{p_0}(p)}{|\nabla \vf_{p_0}(p)|},
\]
where $\nabla$ denotes the Riemannian gradient.

\begin{definition}\label{D:char}
	Let  $\Om \subset \G$ be an open set of class $C^1$. A point $p_0\in \pa\Om$ is called \emph{characteristic} if
	\begin{equation}\label{cp}
	\nu(p_0) \perp \cH_{p_0},
	\end{equation}
	where $\cH_{p_0}$ is as in \eqref{fiber}.
	The \emph{characteristic set} $\Sigma = \Sigma_{\Om}$ is the collection of all characteristic points of $\Om$. A boundary point $p_0 \in \pa \Om \setminus \Sigma$ is called  \emph{non-characteristic} boundary point. For more details, we refer to \cite{CG}.
\end{definition}

\subsection{Modulus of continuity and its properties}
\begin{definition}\label{kartik}
A function $\Phi(s)$ for $0\leq s\leq R_0$ is called the modulus of continuity if the following properties are satisfied
\begin{enumerate}
\item{}~$\Phi(s)\rightarrow0$~~~as~~$s\rightarrow0.$
\item{}~$\Phi(s)$ is positive and increasing as a function of $s.$
\item{}~~$\Phi$ is sub-additive, i.e, $\Phi(s_1+s_{2})\leq \Phi(s_1)+\Phi(s_2)$
\item{}~~$\Phi$~ is continuous.
\end{enumerate}
\end{definition}
Let us define the notion of Dini-continuity.
\begin{definition}
Suppose that $\Omega\subset\mathbb{G}$ and $f:~\Omega\longrightarrow\R$ is a given function. Then we define the modulus of continuity of $f$ as follows:
\begin{equation}
\omega_{f}(s)=\sup_{d(p,\overline{p})\leq s}|f(p)-f(\overline{p})|.
\end{equation}
We say that the function $f$ is Dini-continuous if
\begin{equation}\label{di}
\int^{1}_{0}\frac{\omega_{f}(s)}{s}ds<\infty.
\end{equation}
\end{definition}
Notice that for a continuous function $f,$ $\omega_{f}$ satisfies all properties (1)-(4) mentioned in Definition \ref{kartik}. Similarly, for a vector valued function $(f_1,f_2,\cdots,f_{m}):\Omega\longrightarrow\R^{m}$ we define the modulus of continuity as follows:
\begin{equation}
\omega_{f}(s)=\sup_{d(p,\overline{p})\leq s}|f(p)-f(\overline{p})|.
\end{equation}
So, as above the function $(f_1,f_2,\cdots,f_m)$ is called Dini-continuous if \eqref{di} holds. From [\cite{GG}, Page 44], we see that any continuous, increasing function $\Phi(s)$ on the interval $[0,R_0]$ which satisfies $\Phi(0)=0$ is modulus of continuity if it is concave. From this, we have  the following important result proved in [\cite{GG}, Theorem 8]:
\begin{theorem}\label{KA2.5}
For each modulus of continuity $\Psi(s)$ on $[0,R_0],$ there is a concave modulus of continuity $\tilde{\Phi}(s)$ with the property
\begin{equation}
\Psi(s)\leq \tilde{\Psi}(s)\leq 2\Psi(s)~~~\text{for~all}~~~s\in[0,R_0].
\end{equation}
\end{theorem}
\begin{definition}\label{alphad}
Given $\alpha>0,$ we say that the modulus of continuity $\Psi$ is $\alpha$ decreasing if for any $t_1,t_2\in(0, R_0]$ satisfying $t_1\leq t_2,$ we have
$$\frac{\Psi(t_1)}{t^{\alpha}_1}\geq \frac{\Psi(t_2)}{t^{\alpha}_2}.$$
\end{definition}

\subsection{Some known results}
\noindent
The first result of this subsection is the extension lemma. This will be used in the proof of the compactness lemma below.
\begin{lemma}\label{extl}
	Let $k_0 \in \mathbb{N}$ be a fixed integer and let $\Om$ be a $C^{k_{0},\beta}$ domain, $~f\in \Tau^{{k_0,\beta}}(B(p,1)\cap \Om)$ be a  function for some fixed $p \in \pa \Om$ and $\beta>0.$ There exists a $\Tau^{{k_0,\beta}}$ function $\tilde{f}$ defined on $B(p,1)$ such that $\tilde{f}(\overline{p})=f(\overline{p})$ whenever $\overline{p}\in B(p,1)\cap \Om$
	and
	$$\|\tilde{f}\|_{\Tau^{{k_0,\beta}}(B(p,1))}\leq C'\ \|f\|_{\Tau^{{k_0,\beta}}(B(p,1)\cap \Om)}.$$
\end{lemma}
\begin{proof}
	One can find the proof in \cite{BG}. For the sake of completeness, we briefly present the proof.	
	Let $\tilde{p}=\Phi(p)$ be the $C^{k_0,\beta}$ local diffeomorphism that straightens the portion $\mathscr S_1$ of $\pa \Om.$ In fact, we locally expressed $\Phi$ in logarithmic coordinates as
	\[
	\Phi(x,y)= (x',x_m - \psi(x',y),y).
	\]
	Let  $v(\tilde{p}) = f\circ\Phi^{-1}(\tilde{p})$ and we write $\tilde{p}=\left(\tilde{x}',\tilde{x}_m,\tilde{y}\right)$ the logarithmic  coordinates of $\tilde{p}$. The function $v$ is now defined for $\tilde x_m \ge 0$. Then, by the classical method of extension in term of reflection, we define the extension of $v$ to the region  $\{\tilde{x}_m<0\}$ as follows:
	\begin{equation} \label{ext11}
	V(\tilde{x}',\tilde{x}_m,\tilde{y})=
	\begin{cases}
	v(\tilde{x}',\tilde{x}_m,\tilde{y})\ \ \ \ \ \ \ \ \ \ \ \ \ \ \ \ \tilde x_m \ge 0,
	\\
	\sum^{k_0+1}_{i=1}c_iv(\tilde{x}',-\frac{\tilde{x}_m}{i},\tilde{y})\ \quad \tilde{x}_m<0,
	\end{cases}
	\end{equation}
	where the constants $c_i,~i=1,\dots,k_0+1$ are determined by the system of equations
	\begin{equation}
	\label{constants1}
	\sum^{k_0+1}_{i=1}c_i(-1/i)^m=1,\quad m=0,1, \dots, k_0,
	\end{equation}
	see e.g. p. 14 in \cite{LM}. We now define the extension $\tilde{f}$ of $f$ in the following way $$\tilde{f}=V\circ \Phi.$$
	It is easy to see that the extension function $\tilde{f}\in \Tau^{{k_0,\beta}}(B(p,1)\cap \Om)$ and the following bound holds
	$$\|\tilde{f}\|_{\Tau^{{k_0,\beta}}(B(p,1))}\leq C'\ \|f\|_{\Tau^{{k_0,\beta}}(B(p,1)\cap \Om)}.$$
	This completes the proof of the Lemma.
\end{proof}

Next, we recall the following smoothness result at the non-characteristic portion of the boundary, see Theorem 3.5 \cite{BG}.
\begin{theorem}\label{KNc1}
	Let $\mathbb{A}=[a_{ij}]$ be a symmetric constant-coefficient matrix. Assume that $\Om$ be a $C^\infty$ domain, and let $u\in \mathscr{L}^{1,2}_{loc}(\Om)\cap C(\overline{\Om})$ be a weak solution of \eqref{dp1} with $f_i, g \equiv 0$.  Let $p_0\in\pa\Om$ be a non-characteristic point and assume that  for some neighborhood $W=B_\Ri(p_0, r_0)$ of $p_0$, we have that   $u \equiv 0$ in $\pa \Om \cap W$.  Then there exists an open neighborhood $V$ of $p_0$ depending on $W$  and $\Om$  and a positive constant $C^\star=C^\star(M, p_0)>0$, depending on $p_0$ and $M= \underset{\Om}{\sup}\ |u|$, such that
	\begin{equation}\label{c2u1}
	\|u\|_{C^2(\overline{\Om}\cap V)}\leq C^\star.
	\end{equation}
\end{theorem}
Next, we state a H\"older continuity result near a $C^{1,\text{Dini}}$ non-characteristic portion of the boundary that is direct consequence of the results in \cite{D}.
\begin{proposition}\label{hol1}
	Let $\Om\subset \mathbb{G}$ be a $C^{1,\text{Dini}}$ domain such that $p_0  \in \pa \Om$ is a non-characteristic point.
	Suppose $u\in \mathscr{L}^{1,2}_{loc}(\Om)\cap C(\overline{\Om})$ be a weak solution of
	\begin{equation}
	\begin{cases}
	\sum_{i,j=1}^m X_i^\star(a_{ij} X_j u) = \sum_{i=1}^m X_i^\star  f_i + g,
	\\
	u=h\ \text{on}\ \partial \Omega,
	\end{cases}
	\end{equation}
	where $\mathbb{A}=[a_{ij}]$ is a symmetric matrix satisfying \eqref{ea01}, for all $p \in \Om.$ Furthermore, assume that  $f^i \in L^{\infty}(\Om),~g\in L^q(\Om),~Q<2q<2Q$ and $h \in \Tau^{0,\gamma}( \pa \Omega)$ for some $\gamma>0.$ Then, there exist $r_0, C>0$ and  $\beta \in (0,1)$, depending on $\Om$, $\lambda$, $\gamma$ and $M \overset{def}{=}\underset{\Om}{\sup}\ |u|<\infty$, such that
	\begin{equation}\label{holder1}
	\underset{\underset{p\neq p'}{p ,p'\in\overline{\Om \cap B(p_0,r)}}}{\sup} \frac{|u(p)-u(p')|}{d(p,p')^{\beta}}\leq C.
	\end{equation}
\end{proposition}

\section {Proof of main result} \label{sec3}
In this section, we will prove our main result, Theorem \ref{main}.
Given a bounded open set $\Omega\subset \G,$ with $p_0\in \partial \Omega$ we will use the notations $\V_\tau$ and $\mathscr S_\tau$ as in \eqref{VS}.
The proof of the Theorem \ref{main} follows in several steps. The first step is to establish the compactness lemma. In the proof of the compactness lemma we need the following  Caccioppoli type inequality. This type of inequality has different applications in the PDE's. So we are presenting it as an independent result.
\begin{lemma}\label{energy}
	Suppose that \eqref{ea01} hold. Let $u\in \mathscr{L}^{1,2}_{loc}(\V_1) \cap C(\overline{\V_1})$ be a weak solution to \eqref{dp1} in $\V_1$ with $\|u\|_{L^{\infty}(\V_1)} \leq 1$. Furthermore, assume that  $f^i \in L^{\infty}(\Om),~g\in L^q(\Om), ~Q<2q$ and there is an $R>0$ such that $B(p,2R)\subset\V_1,$ then the following estimate holds:
	\begin{equation}\label{en}
	\int_{B(p,R)}|\nabla_\cH u|^2  \leq   C \left[\sum_{i=1}^m \|f^i\|_{L^{\infty}(B(p,2R))}+ \|g\|_{L^{q}(B(p,2R))}\right],
	\end{equation}
	for some universal $C(Q, \lambda).$
\end{lemma}
\begin{proof}	
	Let $\phi$ be a smooth cut off function such that $\phi \equiv 1$ in $B(p,R)$ and vanishes outside $B(p,2R)$.
	Now by taking  $\eta= \phi^2 u$ as a test function in the weak formulation, we obtain the following equality
	\begin{eqnarray*}
		\int_{B(p,2R)}\phi^2 \langle A \nabla_\cH u, \nabla_\cH u \rangle&=&\int_{B(p,2R)}\phi^2 \, \langle f, \nabla_\cH u \rangle +2\int_{B(p,2R)}\phi u \, \langle f, \nabla_\cH \phi \rangle\\
		&-&\int_{B(p,2R)} g \phi^2u -2\int_{B(p,2R)}\phi u \langle A \nabla_\cH u, \nabla_\cH \phi \rangle,
	\end{eqnarray*}
	where $f=(f_1,\dots,f_m).$ Now, by applying Cauchy Schwartz  inequality and the fact that $\|u\|_{L^{\infty}(\V_1)} \leq 1,$  we obtain
	\begin{eqnarray} \label{e8}
	&&\lambda \int_{B(p,2R)} \phi^2 |\nabla_\cH u|^2 \leq C\bigg[ \sum_i\|f_i\|_{L^{\infty}(B(p,2R))} \|\phi\|^2_{L^{2}(B(p,2R))} \nonumber\\
	&&+\frac{\lambda}{2} \int_{B(p,2R)} \phi^2 |\nabla_\cH u|^2+\|\nabla_\cH \phi\|_{L^{2}(B(p,2R))}+\|g\|_{L^{q}(B(p,2R))} \|\phi\|^2_{L^{2q/(q-1)}(B(p,2R))}\bigg].
	\end{eqnarray}
	By subtracting off the second integral in the right hand side of \eqref{e8} from the left hand side in \eqref{e8}, we obtain that the desired conclusion follows by using bounds on $\phi$ and the fact that $\phi \equiv 1$ in $B(p,R)$.	
\end{proof}
\subsection{Compactness lemma}
Now, we ready to prove the compactness Lemma \ref{compactness}. This lemma states that if the coefficient matrix $[a_{ij}]$ in \eqref{dp1} is very close to the constant matrix in certain norm and the other data are sufficiently small then the solutions of \eqref{dp1} can be approximated by a sufficiently smooth functions, in fact by the solutions of uniformly elliptic equation with constant coefficient.
\begin{lemma}\label{compactness}
	Suppose that \eqref{ea01} hold. Assume that for a given $p_0 = e\in \pa \Om$ the set $\mathscr S_1$ be non-characteristic, and that in the logarithmic coordinates $\V_1$ is given by $\{(x,y)\mid x_m > \psi(x',y)\}$, where $\psi \in C^{1,\text{Dini}}$, and $x'=(x_1,..., x_{m-1}).$ Let $u\in \mathscr{L}^{1,2}_{loc}(\V_1) \cap C(\overline{\V_1})$ be a weak solution to \eqref{dp1} in $\V_1$ with $\|u\|_{L^{\infty}(\V_1)} \leq 1.$
	Then, for a given $\ve >0$ there exists $\delta = \delta(\ve)>0$ such that if
	\begin{equation}\label{dilf1111}
	\|\psi\|_{C^{1,\text{Dini}}}\leq \delta,\ ||a_{ij}-a^{0}_{ij}||_{L^{\infty}(\V_1)}\leq \delta,\ ||h||_{\Tau^{0,\alpha}(\mathscr S_1)}\leq \delta,\ ||f_{i}||_{L^{\infty}(\V_1)}\leq \delta,\ ||g||_{L^{q}(\V_1)} \leq \delta,
	\end{equation}
	we can find  $w\in C^2( \overline{\V_{1/2}})$  such that
	\[
	\|u-w\|_{L^{\infty}(\V_{1/2})}\leq \ve,
	\]
	with
	\[
	\|w\|_{C^2(\overline{\V_{1/2}})}\leq C C^\star.
	\]
	Here, the constant $C>0$ is a universal constant, whereas $C^\star$ can be taken as that in the estimate \eqref{c2u1} in Lemma \ref{KNc1}, corresponding to $p_0 = e$ and $M=1$.
\end{lemma}
\begin{proof}
	The proof of lemma follows by the standard contradiction argument as in the work \cite{Ca}. Suppose that there exists an $\ve_0>0$ such that for every  $\nu\in\mathbb{N}$ we can find:
	\begin{enumerate}
		\item a matrix-valued function $\mathbb A^\nu = [a^\nu_{ij}]$ with continuous entries in $\G$ and satisfying \eqref{ea01},
		\item a domain $\Om_\nu$ with ${\V}^\nu_1 = \Om_\nu\cap B(1)$ and $\mathscr S^\nu_1 = \pa \Om_\nu \cap B(1),$
		\item a solution $u_\nu$ to the problem
		\begin{equation}\label{cacc1}
		\sum_{i,j=1}^m X_i^\star (a^\nu_{ij}X_j u_\nu)  = \sum_{i=1}^m X_i^\star f^\nu_i + g_\nu\ \ \ \text{in}\ \V^\nu_1,\ \ \ \ \ u_\nu  = h_\nu\ \ \ \text{on}\ \mathscr S^\nu_1,
		\end{equation}
	\end{enumerate}
	along with
	\begin{equation}\label{uk1}
	\begin{aligned}{}
	&\|u_\nu\|_{L^{\infty}(\V^\nu_1)}\leq 1,\\
	&\|\psi_\nu\|_{C^{1,\text{Dini}}} \leq \frac 1\nu,~||a^\nu_{ij}-a^{0}_{ij}||_{L^{\infty}(\V^\nu_1)} \leq \frac 1\nu,~||h_\nu||_{\Tau^{0,\alpha}(\mathscr S_1^\nu)}\leq \frac 1\nu,\\
	&||f^\nu_{i}||_{L^{\infty}(\V^\nu_1)}\leq \frac 1\nu,~ ||g_\nu||_{L^{q}(\V^\nu_1)} \leq \frac 1\nu,
	\end{aligned}
	\end{equation}
	but for every $w\in C^2(\overline{\V^\nu_{1/2}})$ and  $\|w\|_{C^2(\overline{\V^\nu_{1/2}})} \le C C^\star$ we have
	\begin{equation}\label{contr1}
	\|u_\nu-w\|_{L^{\infty}(\V^\nu_{1/2})}\geq \ve_0.
	\end{equation}
	Note that the sets $\V^\nu_1$ above are described in the logarithmic coordinates by the functions $\psi_\nu\in C^{1,\text{Dini}}$, that is, $\{(x,y)\mid x_m> \psi_\nu(x',y)\}$. Now, we will show that the validity of \eqref{contr1} leads to a contradiction. We proceed by observing that the uniform bounds in \eqref{uk1} combined with Proposition \ref{hol1}, produces constants $C, \beta>0,$ depending on $\lambda, \alpha$, but not on $\nu,$ such that
	\[
	\|u_\nu\|_{\Tau^{{0,\beta}}(\V^\nu_{4/5})}\leq C.
	\]
	Since $u_\nu$'s are defined on varying domains $\V^\nu_1,$ we need to work with functions defined on same domain. To do this,
	we now use an idea similar to that in the proof of \cite[Lemma 4.1]{BG}. After flattening the boundary as in Lemma \ref{extl}, we extend $u_\nu$ to $B(1)$ using \eqref{ext11} and denote the extended function by $U_\nu.$  By Lemma \ref{extl} with $k_0=0$, it is easy to see that such an extension ensures that $U_\nu$ is uniformly bounded in $\Tau^{{0,\beta}}(B(\frac{4}{5}) ).$
	As a consequence, we have the following convergence results.
	\begin{enumerate}
		\item By  applying Arzela-Ascoli theorem we obtain a subsequence, that we will still denote by $\{U_\nu\}_{\nu\in \mathbb N}$, that converges uniformly to a function $U_0\in \Tau^{{0,\beta}}(B(4/5))$. Clearly, $U_0$ satisfies
		\begin{equation}\label{extw1}
		U_0(x',x_m,y)=
		\begin{cases}
		U_0(x',x_m,y)\ \ \ \ \ \ \ \ \ \ \ \ \ \ \ \ \ \ \ \ \ \ \ \ \ \ x_m\geq 0,\\
		\sum^3_{i=1} c_i U_0(x',-x_m/i,y)\ \ \ \ \ \ \ \ \ \ \ x_m< 0,
		\end{cases}
		\end{equation}
		where the constants $c_1,c_2$ and $c_3$ are given by the system (\ref{constants1}).
		\item From \eqref{uk1}, we see that $f_\nu \to 0$ as $\nu \to \infty$.
		\item Since by \eqref{uk1} we have  $||\psi_\nu||_{\Tau^{0,1}(\V^\nu_1)} \leq \frac{1}{\nu}$ for every $\nu$ so we  get
		\begin{equation}\label{zero1}
		U_0(x', 0, y)=0.
		\end{equation}
	\end{enumerate}	
	Now, we will show that $U_0 \in \mathscr{L}^{1,2}_{loc}(B(4/5) \cap \{x_m>0\}) \cap C(\overline{B(4/5) \cap\{x_m>0\}}).$ Moreover $U_0$ is a weak solution to the problem
	\begin{equation}\label{weq1}
	\sum_{i,j=1}^m a^{0}_{ij}X_iX_j U_0 = 0\ \ \text{in}\ B(4/5) \cap \{x_m>0\},\ \ \ \ \
	U_0=0\ \  \text{on}\ B(4/5)\cap \{x_m=0\}.
	\end{equation}
	To see this, let us observe that $\|\psi_\nu\|_{C^{1,\text{Dini}}}\le 1/\nu\to 0,$ so for a given $p \in B(4/5) \cap \{x_m>0\},$  there exist $\eta>0$ and $\nu_0(p)\in \mathbb N$ such that for all $\nu\geq \nu_0(p)$ we have $B(p,2\eta)\subset \V^\nu_1.$ By the Caccioppoli inequality (see Lemma \ref{energy} with $R=\eta$) for the problem \eqref{cacc1} combined with the uniform bounds in \eqref{uk1}, we find that for all $\nu\geq\nu_0(p)$ following inequality holds:
	\begin{equation}\label{ue11}
	\int_{B(p,\eta)}|\nabla_\cH u_\nu|^2 \leq C,
	\end{equation}
	for some $C(\lambda,\eta)>0$ independent of $\nu.$ Therefore, $\{u_\nu\}_{\nu\in \mathbb N}$ has a subsequence, which we still denote by $\{u_\nu\}_{\nu\in \mathbb N}$, such that
	\[
	u_\nu\rightarrow w\ \ \text{weakly in}\ \mathscr{L}^{1,2}(B(p,\eta)),\ \  \text{and}\ \ \
	u_\nu\rightarrow w\ \  \text{strongly in} \ \ L^2(B(p,\eta)).
	\]
	Since $\{U_\nu\}_{\nu\in \mathbb N}$ converges to $U_0$ uniformly, by uniqueness of limits we can assert that $w=U_0$ in $B(p, \eta)$. Moreover, using the uniform energy estimate for the $u_\nu'$s in \eqref{ue11} and \eqref{uk1} it follows by standard weak type arguments that $U_0$ is a weak solution to
	\[
	\sum_{i,j=1}^m a^0_{ij}X_i X_j U_0=0
	\]
	in $B(p, \eta),$ and hence a classical solution by H\"ormander's hypoellipticity theorem in \cite{H}. By the arbitrariness of $p \in B(4/5) \cap \{x_m>0\}$ and \eqref{zero1}, we conclude that \eqref{weq1} holds.
	
	We can now make use of the estimate from  Theorem \ref{KNc1} to obtain
	\[
	\|U_0\|_{C^2(\overline{B(1/2) \cap \{x_m>0\}})}\leq C^\star
	\]
	for some universal $C^\star>0$.	This follows since $[a^0_{ij}]$ is a constant coefficient matrix, and the portion $B(4/5)\cap \{x_m=0\}$ of the boundary of $B(4/5) \cap\{x_m>0\}$ is non-characteristic and $C^\infty.$ Now, from  the expression of $U_0$ in \eqref{extw1} we see that the second derivatives in $x_m$ are continuous across $x_m = 0$, and thus in fact $U_0 \in C^2(B(1/2))$, and
	\[
	\|U_0\|_{C^2(\overline{\V^\nu_{1/2}})} \leq \|U_0\|_{C^2(B(1/2))}\leq C C^\star,
	\]
	where $C>0$ is a universal constant. This shows that $w = U_0$ is an admissible candidate for the estimate \eqref{contr1}.  In particular, we have for $\nu\in\mathbb N$
	\[
	0 < \ve_0 \le \|u_\nu- U_0\|_{L^{\infty}(\V^\nu_{1/2})},
	\]
	which is a contradiction for large enough $\nu$'s, since $u_\nu \to U_0$ uniformly. This completes the proof of the lemma.
\end{proof}
Having proved the compactness lemma, now we are ready to prove main Theorem \ref{main}. Since proof of the theorem is long so we have divided it in many steps.
\begin{proof}{of Theorem \ref{main}.} We divide the proof into five steps:
	\begin{enumerate}
		\item Preliminary reductions:
		\item Setting modulus of continuity.
		\item  Existence of the first order Taylor polynomial at every $\overline{p}\in \mathscr S_{1/2}.$
		\item Continuity of the horizontal gradient on $\mathscr S_{1/2}.$
		\item Patching the interior and boundary estimate (Modulus of continuity of the horizontal gradient upto the boundary).
	\end{enumerate}
	\noindent \textbf{(1)~~Preliminary reductions}~Let us make some observations.
	
	\noindent\textbf{(a)}~First we consider $\hat{u}=u-h$ which solves:
	\begin{equation}
	\sum_{i,j=1}^m X_i^\star (a_{ij}X_j \hat{u})= \sum_{i=1}^m X_i^\star \hat{f}_i + g\ \ \ \text{in}\ \V_{\tau},\ \ \ \ \ \hat{u}  =0\ \ \ \text{on}\ \mathscr S_{\tau},
	\end{equation}
	where $\hat{f}_{i}=f_{i}-\sum^{m}_{j=1}a_{ij}X_{j}h,$ which is again Dini continuous with the modulus of continuity depending on the modulus of continuity of $A=[a_{ij}],$ $h$ and $f_{i}.$ More precisely, for any $p,q\in\Om$  we have:
	\begin{equation*}
	|\hat{f}_{i}(p)-\hat{f}_{i}(q)|\leq \,\,\omega_{f_{i}}(d(p,q))+\|A\|_{L^{\infty}(\Omega)}\,\,\omega_{\nabla_\cH h}(d(p,q))+\|\nabla_\cH h\|_{L^{\infty}(\Omega)}\,\,\omega_{A}(d(p,q)).
	\end{equation*}
	Therefore, $\hat{f}_{i}$'s are Dini continuous functions and hence, without loss of generality, we can assume that $h\equiv0.$
	
	\noindent\textbf{(b)}~In view of the left translation we may assume that $p_{0}=e.$ Furthermore, by scaling with respect to the family of dilations $\{\delta_{\lambda}\}_{\lambda>0}$ and suitable rotation of the horizontal layer $\g_1,$ without loss of generality we may assume that
	\begin{enumerate}
		\item{}~$\tau=1.$
		\item{}~~$p_{0}=e.$
		\item{}~~In the logarithmic coordinates, $\V_{1}=\Om\cap B(1)$ can be expressed as
		\begin{equation}
		\big\{(x',x_{m},y)~~~|~~~x_{m}>\psi(x',y)\big\}
		\end{equation}
		with $\psi(0,0)=0,$ $\nabla_{x'}\psi(0,0)=0$ and $\|\psi\|_{C^{1,\text{Dini}}}\leq1.$
	\end{enumerate}
	\noindent\textbf{(c)}~~In view of the scaling we may assume that the data are sufficiently small (satisfying \eqref{remark2}), so that we can employ Lemma \ref{compactness}. Indeed, for every $0<\tau\leq1$ consider the domain $\Om_{\tau}=\delta_{\tau^{-1}}(\Om).$ In the logarithmic coordinates $\Om_{\tau}$ can be expressed as follows:
	\begin{equation}\label{3.13}
	\Om_{\tau}=\{(x',x_{m},y_{2},y_{3},\cdots y_{k})~~|~~(\tau x',\tau x_{m}, \tau^{2}y_{2},\cdots\tau^{k} y_{k})\in\Omega\}.
	\end{equation}
	Observe that $\partial\Om_{\tau}$ is given by:
	\begin{equation}
	x_{m}=\psi_{\tau}(x',y)=\psi_{\tau}(x', y_{2},\cdots y_{k}):=\frac{1}{\tau}\psi(\tau x', \tau^{2}y_{2},\cdots \tau^{k}y_{k}).
	\end{equation}
	We set
	\[\mathscr W_{\tau}=\Om_{\tau}\cap B(\tau^{-1}),\,\,\,\,\,\,\mathscr T_\tau=\partial\Omega_{\tau}\cap B(\tau^{-1}).\]
	Let us observe that:
	\begin{equation*}
	\left\{
	\begin{aligned}{}
	\nabla_{x'}\psi_{\tau}(x',y)&=\nabla_{x'}\psi(\tau x', \tau^{2}y_{2},\cdots \tau^{k}y_{k})\\
	\nabla_{y_{j}}\psi_{\tau}(x',y)&=\tau^{j-1}\nabla_{y_{j}}\psi(\tau x',\tau^{2}y_{2},\cdots\tau^{k}y_{k}),\,\,\,\,\text{for}\,\,\,j=2\cdots k.
	\end{aligned}
	\right.
	\end{equation*}
	Thus, $\nabla\psi_{\tau}(x',y)\rightarrow\langle\nabla_{x'}\psi(0,0),0\rangle$ as $\tau\rightarrow0.$ Therefore, by Taylor's theorem we get
	\begin{equation}
	\psi_{\tau}(x',y)\rightarrow\langle\nabla_{x'}\psi(0,0),x'\rangle=0~~~~~\text{as}~~~\tau\rightarrow0,
	\end{equation}
	consequently,
	\begin{equation}
	\partial\Om_{\tau}\cap B(1)\longrightarrow\{x_{m}=0\}\cap B(1).
	\end{equation}
	It is also easy to see that for any $(x',y),(\bar{x'},\bar{y})\in\Om_{\tau}\cap B(1),$ we have:
	\begin{equation}\label{psi}
	\begin{aligned}{}
	&|\nabla\psi_{\tau}(x',y)-\nabla\psi_{\tau}(\bar{x'},\bar{y})|\\
	&\leq (1+\tau+\tau^{2}+\cdots+\tau^{k-1})\omega_{\nabla\psi}(\tau|x'-\bar{x'}|+\cdots+\tau^{k}|y_{k}-\bar{y_{k}}|)\rightarrow0,
	\end{aligned}
	\end{equation}
	as $\tau\rightarrow0.$ In addition, we also observe that $u_{\tau}(p)=u(\delta_{\tau}p)$ solves the following problem:
	\begin{equation}
	\sum_{i,j=1}^m X_i^\star (a_{ij,\rho}X_ju_{\tau})= \sum_{i=1}^m X_i^\star f_{i,\tau} + g_{\tau}\ \ \ \text{in}\ \mathscr W_{\tau},\ \ \ \ \ \ u_{\tau}=0\ \ \ \text{on}\ \mathscr S_{\tau},
	\end{equation}
	where
	\[a_{ij,\tau}(p)=a_{ij}(\delta_{\tau}p),~\,\,\,\,f_{i,\tau}(p)=\tau f_{i}(\delta_{\tau}p),~~~g_{\tau}(p)=\tau^{2}g(\delta_{\tau}p)\,\,\,h_{\tau}(p)=h(\delta_{\tau}p).\]
	Consequently, we have the following relations:
	
	\begin{enumerate}
		\item $|a_{ij,\tau}(p)-a_{ij,\tau}(q)|=|a_{ij}(\delta_{\tau}p)-a_{ij}(\delta_{\tau}q)|\leq\omega_{A}(\tau d(p,q))\rightarrow0\,\,\,\,\text{as}\,\,\,\tau\rightarrow0.$
		\item \hspace{3cm}~$\|f_{i,\tau}\|_{L^{\infty}(((\mathscr W_{\tau})))}\leq \tau\|f_{i}\|_{L^{\infty}(((\mathscr W_{1})))}.$
		\item \hspace{3cm}~$|f_{i,\tau}(p)-f_{i,\tau}(q)|=\tau|f_{i}(\delta_{\tau}p)-f_{i,}(\delta_{\tau}q)|
		\leq\tau\omega_{f}(\tau d(p,q))\rightarrow0\,\,\,\,\text{as}\,\,\,\tau\rightarrow0.$
		\item \hspace{3cm}~$\|g_{\tau}\|_{L^{q}(\mathscr W_{\tau})}=\tau^{2-\frac{Q}{q}}\|g\|_{L^{q}((\mathscr W_{1}))}.$
		\item \hspace{3cm}~$\|\nabla_\cH h_{\tau}\|_{L^{\infty}(\mathscr W_{\tau})}\leq\tau\|\nabla_\cH h_{\tau}\|_{L^{\infty}(\mathscr W_{1})}.$
		\item \hspace{3cm}~~$|\nabla_\cH h_{\tau}(p)-\nabla_\cH h_{\tau}(q)|\leq \tau\omega_{\nabla_\cH h}(\tau d(p,q)).$
	\end{enumerate}
	
	\begin{remark}\label{remark1}
		In view of \eqref{psi} and the above relations, it is clear that by choosing $\tau$ sufficiently small, say $\tau_{0},$ we can make all the data sufficiently small so that the compactness lemma is applicable provided we consider $u_{\tau},$ $a_{ij,\tau},$  $f_{i,\tau}$ $g_{\tau},$ $h_{\tau}$ and $\Omega_{\tau}$ instead of corresponding terms $u,$ $a_{i,j},$  $f_{i}$ $g,$ $h$ and $\Omega.$  Therefore, without loss of generality, from here onwards in the proof of this theorem we assume that
		\begin{eqnarray}\label{remark2}
		&&\hspace{-.9in}\|a_{ij}-a_{ij}(e)\|_{L^{\infty}(\Omega)\cap B(1)}\leq\tilde{\delta},~~\|\psi\|_{C^{1,\text{Dini}}}\leq\tilde{\delta},~~\|h\|_{\Tau^{0,\alpha}(\mathscr S_{1})}\leq\tilde{\delta},~~~~\|f_{i}\|_{L^{\infty}(\mathscr W_{1})}\leq\tilde{\delta} \nonumber\\
		&&\text{and}~~\|g\|_{L^{q}(\mathscr W_{1})}\leq\tilde{\delta}.
		\end{eqnarray}
		where $\tilde{\delta}$ is given by \eqref{tilde{delta}}.
	\end{remark}
	\noindent \textbf{(2)~~Setting modulus of continuity.}
	Let us first fix a constant $\alpha$ (see also Corollary \ref{intr1}) such that $0<\alpha<1$ and consider the function
	\begin{equation}
	\tilde{\omega}_{1}(\sigma)=\max\{\omega_{\nabla\psi}(\sigma),\sigma^{\alpha}\}.
	\end{equation}
	After normalization and using Theorem \ref{KA2.5}, we can assume that $\tilde{\omega}_{1}$ is concave and $\tilde{\omega}_{1}(1)=1.$ With the help of the above function we can define a new function $\omega_{1}(\sigma)=\tilde{\omega}_{1}(\sigma^{\alpha}).$  Then this function becomes $\alpha$ decreasing (see Definition \ref{alphad}) and $\omega_{1}$ is still Dini continuous, for details see \cite{KA}. Now, let us define
	\begin{equation}\label{m2}
	\tilde{\omega}_{2}(\sigma)=\max\{\sigma^{\alpha},\omega_{f}(\sigma)\}.
	\end{equation}
	Again following the similar argument as above for $\omega_{1},$ without loss of generality we can assume that $\tilde{\omega}_{2}$ concave and $\alpha$ decreasing. Having defined $\tilde{\omega}_{2},$ let us define a new function
	\begin{equation}
	\omega_{2}(\sigma)=:\max\Bigg\{C_{II}\sigma\Big(\frac{1}{|\Omega\cap B(\sigma)|}\int_{\Omega\cap B(\sigma)}|g|^{q}\Big)^{\frac{1}{q}},~~~~\tilde{\omega}_{2}(\sigma)\Bigg\}.
	\end{equation}
	Having defined $\omega_{1}$ and $\omega_{2},$ let us define another function as follows:
	\begin{equation}\label{omega1}
	\omega_{3}(\sigma^{l}):=\frac{1}{\tilde{\delta}}\sum^{l}_{j=0}\omega_{1}(\sigma^{l-j})\omega_{2}(\sigma^{j}),
	\end{equation}
	where $\tilde{\delta}$ is given by \eqref{remark2}.
	Finally, let us set
	\begin{equation}\label{omega}
	\omega(\sigma^{l}):=\max\{\omega_{3}(\sigma^{l}),\sigma^{l\alpha}\}.
	\end{equation}
	We will be using some of the properties of the modulus of continuous functions defined above. So for the sake of completeness we list the required properties and sketch their proofs here.
	\begin{lemma} \label{lma3.4}
		\begin{enumerate}
			\item \label{m1}~We have the following estimate
			\begin{equation}\label{1mod}
			\sum^{\infty}_{j=1}\omega(\sigma^{j})\leq C_{b}.
			\end{equation}
			\item~For any fixed positive integer $\nu\in\mathbb{N},$ the following estimate holds:
			\begin{equation}\label{2mod}
			\sigma^{\alpha}\omega(\sigma^{\nu})\leq \omega(\sigma^{\nu+1}).
			\end{equation}
			\item~$\omega_1$ is monotone.
			\item \label{m41}~$1\leq\omega(1).$
			\item~It is also clear that
			\begin{equation}\label{vmod}
			\frac{1}{\tilde{\delta}}\omega_{f_i}(\sigma)\leq \omega(\sigma)~~~ \text{and}~~~\frac{1}{\tilde{\delta}}\omega_2(\sigma^{\nu})\leq \omega(\sigma^{\nu}).
			\end{equation}
			\item \label{m61}~
			\begin{equation}\label{iv}
			\sigma^{\alpha}\leq \omega(\sigma)
			\end{equation}
			\item~$\omega_1$ is $\alpha$-decreasing.
		\end{enumerate}
	\end{lemma}
	We prove \eqref{m1}, \eqref{m61} and rest follows from the definition of the respective modulus of continuity. For details, we refer to Lemmas 4.5 and 4.7 in \cite{KA}.
	\begin{proof} {of \eqref{m1}.} In order to estimate the sum in the left hand side of \eqref{1mod}, we first need to estimate the following sum
		\begin{equation}\label{m4}
		\begin{aligned}{}
		\sum^{\infty}_{j=1}\omega_{3}(\sigma^{j})&=\frac{1}{\tilde{\delta}}\sum^{\infty}_{j=1}\sum^{j}_{i=0}\omega_{1}(\sigma^{j-i})\omega_{2}(\sigma^{i})\\
		&=\frac{1}{\tilde{\delta}}\Big(\sum^{\infty}_{j=0}\omega_{1}{(\sigma^{j})}\Big)\Big(\sum^{\infty}_{j=1}\omega_{2}(\sigma^{j})\Big).
		\end{aligned}
		\end{equation}
		Thus, from \eqref{m4}, it is clear that, in order to estimate the above sum we need to estimate $\sum^{\infty}_{j=1}\omega_{1}(\sigma^{j})$ and $\sum^{\infty}_{j=1}\omega_{2}(\sigma^{j}).$ The sum involving the term $\omega_1$ is finite because of the Dini continuity of $\bigtriangledown\psi.$ More precisely, we have the following estimate:
		\begin{equation}\label{m5}
		\sum^{\infty}_{j=1}\omega_{1}(\sigma^{j})\leq\frac{1}{-\log\sigma}\sum^{\infty}_{j=1}\int^{\sigma^{j-1}}_{\sigma^{j}}\frac{\omega_{1}(t)}{t}dt=\int^{1}_{0}\frac{\omega_1(t)}{t}dt<\infty.
		\end{equation}
		Now, let us estimate the sum involving $\omega_{2}.$ It is easy to see that there exists a constant $\overline{C}$ such that
		\begin{equation}\label{m7}
		\begin{aligned}{}
		\sum^{\infty}_{j=1}\omega_{2}(\sigma^{j})&\leq \overline{C}\int^{2}_{0}\Big(\frac{1}{|\Omega\cap B(\tau)|}\int_{\Omega\cap B(\tau)}|g(x)|^{q}dx\Big)^{\frac{1}{q}}d\tau+\sum^{\infty}_{j=1}\tilde{\omega}_{2}(\sigma^{j})\\
		&=:\overline{C}\tilde{\mathbf{I}_{q}^{g}}(e,2)+\sum^{\infty}_{j=1}\sigma^{\alpha j}+\sum^{\infty}_{j=1}\omega_{f}(\sigma^{j})\\
		&=I+II+III,
		\end{aligned}
		\end{equation}
		where $\tilde{\mathbf{I}_{q}^{g}}$ is defined in \eqref{ig}.
		Note that $I$ is finite because $g\in L(Q,1)$ (see Definition \ref{lorentz1}) so making use of the result in \cite[Equation (3.13)]{M16}, we get
		\begin{equation}\label{m6}
		\sup_{p}\tilde{\mathbf{I}^{g}_{q}}(p,r)\leq\frac{1}{|B(1)|^{\frac{1}{Q}}}\int^{|B(r)|}_{0}\Big[g^{**}(\tau)\tau^{\frac{q}{Q}}\Big]^{\frac{1}{q}}\frac{d\tau}{\tau}.
		\end{equation}
		$II$ is finite because it is geometric sum. While $III$ is finite because $f$ is Dini continuous as in \eqref{m5} the sum containing $\omega_{1}$ is finite. Thus, by using \eqref{m6} (with $f=g$ there), \eqref{m7} and \eqref{m5} in \eqref{m4}, we find that the sum in \eqref{m4} is finite.
	\end{proof}
	
	\begin{proof}{of \eqref{m61}.}
		From \eqref{omega}, if $\omega(\sigma^{\nu})=\sigma^{\nu\alpha},$ then we get
		\begin{equation}
		\sigma^{\alpha}\omega(\sigma^{\nu})=\sigma^{\alpha(1+\nu)}\leq\omega(\sigma^{1+\nu}).
		\end{equation}
		Now suppose that $\omega(\sigma^{\nu})=\omega_{3}(\sigma^{\nu}).$ In this case let us proceed as follows:
		\begin{equation}\label{l4}
		\begin{aligned}{}
		&\sigma^{\alpha}\omega_{3}(\sigma^{\nu})=\frac{1}{\tilde{\delta}}\sum^{\nu}_{j=0}\sigma^{\alpha}\omega_{1}(\sigma^{\nu-j})\omega_{2}(\sigma^{j}) \leq\frac{1}{\tilde{\delta}}\sum^{l}_{j=0}\omega_{1}(\sigma^{1+\nu-j})\omega_{2}(\sigma^{j})~~~~(\text{since}~~\omega_{1}(\cdot)~\text{is}~\alpha~\text{decreasing}) \\
		&\leq\omega{(\sigma^{1+\nu})}~~~(\text{by~definition~}~\omega_{3}) \leq\omega(\sigma^{\nu})~~~\text{by}~\eqref{omega}.
		\end{aligned}
		\end{equation}
	\end{proof}
	\noindent
	\textbf{(3)~Existence of the first order Taylor polynomial at every $\overline{p}\in \mathscr S_{1/2}.$}
	The aim of this section is to establish that $u$ is $\Gamma^{1}(\overline{p})$ for every $\overline{p}\in\mathscr S_{1/2}.$ More precisely, we want to establish the estimate \eqref{1m}, which will be accomplished in two sub steps. In the first sub step we show that for any $\overline{p}\in\mathscr S_{1/2}$ there exists a sequence of first order polynomial approximating $u$ near $\overline{p}.$ Later on in the next substep we show that the limiting polynomial will give the affine approximation to the solution at $\overline{p}.$
	
\noindent	
	\textbf{(a)}~Let $\overline{p}\in\mathscr S_{1/2}$ be a non-characteristic point. In view of translation and rotation without loss of generality we can assume that $\overline{p}=e\in \mathscr S_{1/2}.$ Also by normalizing the solution if necessary, we can assume that $\|u\|_{L^{\infty}(\V_1)}\leq 1.$ Denote the constant $CC^{*}$ in the compactness Lemma \ref{compactness} by $\theta$ and fix $\sigma>0$ such that
	\begin{equation}\label{choicesigma}
	0<\sigma<(4\theta)^{-\frac{1}{1-\alpha}}.
	\end{equation}
	We also let
	\begin{equation}\label{3.37}
	\epsilon=\frac{\sigma^{1+\alpha}}{2}.
	\end{equation}
	Suppose that $\delta(\epsilon)$ be the number in the compactness Lemma \ref{compactness} corresponding to $\epsilon$ defined above. Let us take another number $\tilde{\delta}\in (0,\delta)$ which will be fixed later. In view of the Remark \ref{remark1}, it is clear that by choosing the scaling parameter $\tau$ sufficiently small we may assume that the smallness condition in \eqref{remark2} with such an $\tilde{\delta}$ can be ensured.
	
	For any  $\kappa\in \mathbb{N}\cup \{0\}$, first we denote by $\mathfrak{P}_\kappa$ the set of homogeneous polynomials in $\G$ of homogeneous degree less or equal to $\kappa$.
	Now, we use induction to show that there exists a sequence of polynomials $\{L_{\nu}\}_{\nu\in\mathbb{N}\cup\{-1,0\}}$ in $\mathfrak{P}_{1}$ such that for every $\nu\in\mathbb{N}\cap\{-1,0\}$ following holds:
	\begin{equation}\label{a1}
	\|u-L_{\nu}\|_{L^{\infty}(\Omega\cap B(\sigma^{\nu}))}\leq \sigma^{\nu}\omega(\sigma^{\nu}),
	\end{equation}
	\begin{equation}\label{1b}
	\|L_{\nu+1}-L_{\nu}\|_{L^{\infty}(B(\sigma^{\nu}))}\leq C\sigma^{\nu}\omega(\sigma^{\nu}),
	\end{equation}
	\begin{equation}\label{a11}
	|L_{\nu}|\leq C_{b}\theta~~~(\text{where}~~C_{b}~~\text{is~from}~~\eqref{1mod}),
	\end{equation}
	\begin{equation}\label{1c}
	\|L_{\nu}\circ\delta_{\sigma^{\nu}}\|_{\Tau^{0,1}(\partial\Omega_{\sigma^{\nu}}\cap B(1))}\leq \delta\sigma^{\nu}\omega(\sigma^{\nu}),
	\end{equation}
	where $\Omega_{\nu}=\delta_{\nu^{-1}}(\Omega)$ is defined in \eqref{3.13}. We prove the above assertion by mathematical induction. Let us set $a_{-1}=a_{0}=0$ and by definning the corresponding polynomials $L_{0}=L_{-1}=0$ we get:
	\begin{equation}
	\|u\|_{L^{\infty}(\Omega\cap B(1))}\leq 1\leq w(1)~~~~\text{by Lemma \ref{lma3.4}\eqref{m41}}.
	\end{equation}
	As we want to establish the continuity of the horizontal gradient at the boundary so we consider the polynomial $L_{\nu}$ of the form $L_{\nu}(p)=l_{\nu} x_m,$ where $(x',x_m,y)$ denotes the logarithmic coordinates of $p.$ Thus the result follows for $\nu=-1,0.$ Now, assume that for some fixed $\nu\in\mathbb{N},$ the polynomials $L_{1},L_{2,}\cdots L_{\nu}$ has been constructed satisfying \eqref{a1}-\eqref{1c}. In order to complete the mathematical induction we need to construct $L_{\nu+1}$ such that \eqref{a1}-\eqref{1c} hold for $\nu+1.$ This will be accomplished by using the compactness Lemma \ref{compactness}. Let us consider the following rescaled function
	\begin{equation}\label{rescaled}
	\tilde{u}(p):=\frac{(u-L_{\nu})(\delta_{\sigma^{\nu}}(p))}{\sigma^{\nu}\omega(\sigma^{\nu})},~\text{for}~~~p\in\tilde{\Omega}\cap B(1),
	\end{equation}
	where $\tilde{\Omega}=\Omega_{\sigma^{\nu}}$. It is easy to observe that $\tilde{u}$ satisfies the following problem:
	\begin{equation*}
	\left\{
	\begin{aligned}{}
	\sum_{i,j=1}^m X_i^\star (a_{ij}X_j\tilde{u})&=\sum_{i=1}^m X_i^\star\tilde{\tilde{f_{i}}} + \tilde{\tilde{g}}\,~\text{in}~~\tilde{\Omega}\cap B(1),\\
	\tilde{u}&=\tilde{\tilde{\Phi}}\ \ \ \text{on}\ \partial\tilde{\Omega}\cap B(1),
	\end{aligned}
	\right.
	\end{equation*}
	where
	\begin{equation}
	\tilde{\tilde{f_{i}}}=\frac{\tilde{f_{i}}-\sum^{m}_{j}\tilde{a}_{ij}X_{j}\tilde{L}_{\nu}}{\omega{(\sigma^{\nu})}},\hspace{0.5cm}~\tilde{\tilde{g}}=\frac{\sigma^{\nu}\tilde{g}}{\omega(\sigma^{\nu})},
	\hspace{.5cm}\tilde{\tilde{\phi}}=-\frac{\tilde{L}_{\nu}}{\sigma^{\nu}\omega(\sigma^{\nu})},
	\end{equation}
	and
	\begin{equation}
	\tilde{a}_{ij}(p)=a_{ij}(\delta_{\sigma^{\nu}}p),\,\,\,\tilde{f}_{i}(p)=f_{i}(\delta_{\sigma^{\nu}}p), \,\,\,\tilde{g}(p)=g(\delta_{\sigma^{\nu}}p),\,\,\,\tilde{L}_{\nu}(p)=L_{\nu}(\delta_{\sigma^{\nu}}p).
	\end{equation}
	Since the result follows for $\nu,$ so in view of \eqref{a1}, we have
	\begin{equation}
	\|\tilde{u}\|_{L^{\infty}(\tilde{\Omega}\cap B(1))}\leq1.
	\end{equation}
	It is also easy to observe the following points: Since $L_{\nu}$ is a polynomial of degree 1 so we have
	\begin{equation}
	\sum^{m}_{ij=1}X^{*}_{i}X_{j}(\tilde{a}^{0}_{ij}\tilde{L}_{\nu})=0,
	\end{equation}
	where $a_{ij}^{0}=a_{ij}(e)$ and $\tilde{a}^{0}_{ij}=\tilde{a}_{ij}(e)=a_{ij}(\delta_{\sigma^{\nu}}(e))=a_{ij}(e).$ Consequently,
	\begin{equation}
	X^{*}_{i}(\tilde{a}_{ij}X_{j}\tilde{L}_{\nu})=\sum^{m}_{ij=1}X^{*}_{i}\big((\tilde{a}_{ij}-\tilde{a}^{0}_{ij})X_{j}\tilde{L}_{\nu}\big)
	\end{equation}
	and also
	\begin{equation}
	X^{*}_{i}(\tilde{f}_{i}-\tilde{f}_{i}(e))=X^{*}_{i}\tilde{f}_{i},
	\end{equation}
	since $X^{*}_{i}\tilde{f}_{i}(e)=0.$ Therefore, we find that $\tilde{u}$ satisfies the following equation:
	\begin{equation*}
	\left\{
	\begin{aligned}{}
	\sum^{m}_{i,j=1}X^{*}_{i}\big(\tilde{a}_{ij}X_{j}\tilde{u}\big)&=\sum^{m}_{i=1}X^{*}_{i}F_{i}+\tilde{\tilde{g}}~~\text{in}~~\tilde{\Omega}\cap B(1),\\
	\tilde{u}&=\tilde{\tilde{\phi}}~~\text{on}~~\partial\tilde{\Omega}\cap B(1),
	\end{aligned}
	\right.
	\end{equation*}
	where
	\begin{equation}
	F_{i}=\frac{\tilde{f}_{i}-\tilde{f}_{i}(e)-\sum^{m}_{j=1}(\tilde{a}_{ij}-\tilde{a^{0}_{ij}})X_{j}\tilde{L}_{\nu}}{\omega(\sigma^{\nu})}.
	\end{equation}
	Now, we show that all the hypotheses in the compactness lemma are satisfied. Indeed, let us observe that:
	\begin{equation}
	\tilde{a}^{0}_{ij}=\tilde{a}_{ij}(e)=a_{ij}(\delta_{\sigma^{\nu}}e)=a_{ij}(e).
	\end{equation}
	Thus, we have
	\begin{equation}\label{1h}
	\|\tilde{a}_{ij}-\tilde{a}^{0}_{ij}\|_{L^{\infty}(\tilde{\Omega}\cap B(1))}=\|a_{ij}-a^{0}_{ij}\|_{L^{\infty}(\Omega\cap B(\sigma^{\nu}))}\leq\omega_{a_{ij}}(\sigma^{\nu})\leq \omega_{A}(\sigma^{\nu}).
	\end{equation}
	Therefore, in view of the Remark \ref{remark1} and the discussion in the beginning  of this section we have   $$\|\tilde{a}_{ij}-\tilde{a}^{0}_{ij}\|_{L^{\infty}(\tilde{\Omega}\cap B(1))}\leq \tilde{\delta}.$$ It follows from \eqref{1c} that
	\begin{equation}
	\|\tilde{\tilde{\phi}}\|_{\Tau^{0,\text{Dini}}(\partial\tilde{\Omega}\cap B(1))}\leq \delta.
	\end{equation}
	For any $q\in\mathscr W_{1},$ we have:
	\begin{equation*}
	\begin{aligned}{}
	|F_{i}(q)|=&\frac{|\tilde{f_{i}}(q)-\tilde{f_{e}}-\sum^{m}_{j=1}(\tilde{a}_{ij}(q)-\tilde{a^{0}}_{ij})X_{j}\tilde{L}_{\nu}(q)|}{\omega(\sigma^{\nu})}\\
	\leq&\frac{|f_{i}(\delta_{\sigma^{\nu}}(q))-f_{i}(e)|+\sum^{m}_{j=1}\big|(\tilde{a}_{ij}(q)-\tilde{a^{0}}_{ij})X_{j}\tilde{L}_{\nu}(q)\big|}{\omega(\sigma^{\nu})}\\
	\leq&\frac{\omega_{f_{i}}(\sigma^{\nu})+\sigma^{\nu}\big|(a_{im}(\sigma^{\nu}q)-a_{im}(e))l_{\nu}\big|}{\omega(\sigma^{\nu})},~~~~\text{where}~~L_{\nu}(x)=l_{\nu}x_{m}.\\
	\leq&\frac{\omega_{f_{i}}(\sigma^{\nu})+\sigma^{\nu}\omega_{A}(\sigma^{\nu})|l_{\nu}|}{\omega(\sigma^{\nu})}
	\leq(1+C_b\theta)\tilde{\delta},
	\end{aligned}
	\end{equation*}
	where $C_{b}$ and $\theta$ are from \eqref{1mod} and \eqref{choicesigma}. In concluding the last line we have used \eqref{vmod}, that is, $\frac{1}{\tilde{\delta}}\omega_{f_{i}}(\nu)\leq \omega(\sigma^{\nu}),$ $\omega_{a_{im}}(\sigma^{\nu})\leq \omega_{A}(1)\leq \tilde{\delta},$  $\sigma^{\nu\alpha}\leq\omega(\sigma^{\nu})$ and $\alpha<1.$ So if we choose
	\begin{equation}\label{tilde{delta}}
	\tilde{\delta}<\frac{\delta}{1+C_b\theta}
	\end{equation}
	we get
	\begin{equation}
	\|F_{i}\|_{L^{\infty}(\tilde{\Omega}\cap B(1))}\leq \delta.
	\end{equation}
	Since $\partial\tilde{\Omega}\cap B(1)$ can be expressed as follows:
	\begin{equation}
	x_{m}=\psi_{\sigma^{\nu}}(x',y_{2},\cdots,y_{k})=\frac{\psi(\sigma^{\nu} x',\sigma^{2\nu}y_{2},\cdots,\sigma^{k\nu}y_{k})}{\sigma^{\nu}}.
	\end{equation}
	Let us denote $\psi_{\sigma^{\nu}}$ by $\tilde{\psi}.$ Therefore, for any $p,\overline{p}\in \tilde{\Omega}\cap B(1)$ with $p=(x',x_m,y_2,\cdots,y_k)$  and $\overline{p}=(\overline{x'},\overline{x}_{m},\overline{y}_{2},\cdots,\overline{y}_{k})$ we have
	\begin{equation}
	|\nabla\tilde{\psi}(p)-\nabla\tilde{\psi}(\overline{p})|\leq (1+\tau+\tau^{2}+\cdots\tau^{k-1})\omega_{\nabla\psi}(\tau|x'-\overline{x'}|+\tau^{2}|y_{2}-\overline{y_{2}}|+\cdots\tau^{k}|y_{k}-\overline{y_{k}}|),
	\end{equation}
	where $\tau=\sigma^{\nu}.$ Since $\tau<1$ and $\psi(0,0)=0$ so by  Remark \ref{remark1}, we have
	\begin{equation}
	\|\tilde{\psi}\|_{C^{1,\text{Dini}}}\leq\delta.
	\end{equation}
	Now, let us consider
	\begin{equation*}
	\begin{aligned}{}
	& \hspace{-.25in}\|\tilde{\tilde{g}}\|_{L^{q}(\tilde{\Omega}\cap B(1))}=\Big(\int_{\tilde{\Omega}\cap B(1)}|\tilde{\tilde{g}}(p)|^{q}dq\Big)^{\frac{1}{q}}=\frac{\sigma^{\nu}}{\omega(\sigma^{\nu})}\Big(\int_{\tilde{\Omega}\cap B(1)}|\tilde{g}(p)|^{q}dp\Big)^{\frac{1}{q}}\\ &\leq\frac{\sigma^{\nu}}{\omega(\sigma^{\nu})}\Big(\frac{1}{|\Omega\cap B(\sigma^{\nu})|}\int_{\Omega\cap B(\sigma^{\nu})}|g(p)|^{q}dp\Big)^{\frac{1}{q}} \leq\frac{\omega_{2}(\sigma^{\nu})}{C_{II}\omega(\sigma^{\nu})} \leq \frac{\tilde{\delta}}{C_{II}}~~~\text{in~view~of~}~\eqref{vmod}.
	\end{aligned}
	\end{equation*}
	Therefore, by the compactness Lemma \ref{compactness}, there exists a $v\in C^{2}(B({\frac{1}{2}}))$ such that $\|v\|_{C^{2}(B(\frac{1}{2}))}\leq \theta$ and
	\begin{equation}\label{11}
	\|\tilde{u}-v\|_{L^{\infty}(\tilde{\Omega}\cap B(\frac{1}{2}))}\leq\epsilon.
	\end{equation}
	Moreover, since $v=0$ on $B(4/5)\cap\{x_{m}=0\}$ so by Taylor's formula and the fact that $\|v\|_{C^{2}(B(1/2))}\leq \theta$ there exists $l\in\R$ with $|l|\leq \theta$ such that
	\begin{equation}\label{12}
	\|v-lx_{m}\|_{L^{\infty}(B(\sigma))}\leq \theta\sigma^{2}<\frac{\sigma^{1+\alpha}}{4},
	\end{equation}
	where the last inequality follows from the choice of $\sigma$ in \eqref{choicesigma}. From \eqref{11}, \eqref{12} and the choice of $\epsilon$ (see \eqref{3.37}) along with the triangle inequality we get the following inequality:
	\begin{equation}\label{1e}
	\|\tilde{u}-lx_{m}\|_{L^{\infty}(B(\sigma))}\leq \sigma^{1+\alpha}.
	\end{equation}
	Let us denote by $L(p)=lx_{m}\in\mathfrak{P}_{1}$ so \eqref{1e} implies that
	\begin{equation}\label{sca4}
	\begin{aligned}{}
	\|\tilde{u}-L\|_{L^{\infty}(\tilde{\Omega}\cap B(\sigma))}&=\sup_{p\in\tilde{\Omega}\cap B(\sigma)}\Big|\frac{(u-L_{\nu})(\delta_{\sigma^{\nu}}(p))}{\sigma^{\nu}\omega(\sigma^{\nu})}-L(p)\Big|\\
	&=\frac{1}{\sigma^{\nu}\omega(\sigma^{\nu})}\|u-L_{\nu+1}\|_{L^{\infty}(\Omega\cap B(\sigma^{\nu+1}))},
	\end{aligned}
	\end{equation}
	where
	\begin{equation}\label{1d1}
	L_{\nu+1}(p):=L_{\nu}(p)+\sigma^{\nu}\omega(\sigma^{\nu})L(\delta_{\sigma^{-\nu}}(p)),~~~\text{for}~~p\in \Omega\cap B(\sigma^{\nu+1}).
	\end{equation}
	It follows from \eqref{1e} and \eqref{sca4} that
	\begin{equation}
	\|u-L_{\nu+1}\|_{L^{\infty}(\Omega\cap B(\sigma^{\nu+1}))}\leq \sigma^{\nu+1}\sigma^{\alpha}\omega(\sigma^{\nu})\leq~\sigma^{\nu+1}\omega(\sigma^{\nu+1})~~~(\text{by}~~\eqref{2mod}).
	\end{equation}
	Also, from \eqref{1d1},
	\begin{equation}
	\|L_{\nu+1}-L_{\nu}\|_{L^{\infty}(B(\sigma^{\nu}))}\leq C\sigma^{\nu}\omega(\sigma^{\nu}),
	\end{equation}
	where $C=\|L\|_{L^{\infty}(B(1))}.$ Moreover, from the expression of $L_{\nu+1}$ in terms of $L_{\nu}$ as in \eqref{1d1} we can infer by induction that in the logarithmic coordinates the polynomials $L_\nu$ are of the form
	\begin{equation}\label{1i}
	L_{\nu}(p)=l_{\nu}x_{m},
	\end{equation}
	where
	\begin{equation}\label{1j}
	|l_\nu|\leq\sum^{\nu}_{j=0}\theta\omega(\sigma^{j})\leq\theta\sum^{\infty}_{j=0}\omega(\sigma^{j})\leq C_{b}\theta.
	\end{equation}
	Therefore, \eqref{a11} follows. In order to prove \eqref{1c}, let us consider points $p,\overline{p}\in\partial\tilde{\Omega}\cap B(1),$ where $\tilde{\Omega}=\Omega_{\sigma^{-(\nu+1)}}=\delta_{\sigma^{-(\nu+1)}}\Omega.$ Let $(x,y)$ and $(\overline{x},\overline{y})$ denote the logarithmic coordinates of $p$ and $\overline{p}$ respectively. With $\tau=\sigma^{\nu+1}$ we have
	\begin{equation}
	x_{m}=\frac{\psi(\tau x',\tau^{2}y_{2},\cdots,\tau^{k}y_{k})}{\tau}~~\text{and}~~~\overline{x}_{m}=\frac{\psi(\tau \overline{x}',\tau^{2}\overline{y}_{2},\cdots\tau^{k}\overline{y}_{k})}{\tau}.
	\end{equation}
	This gives
	\begin{equation}\label{sca6}
	\begin{aligned}{}
	|L_{\nu+1}(\delta_{\tau}p)-L_{\nu+1}(\delta_{\tau}\overline{p})|&=|l_{\nu+1}||\tau x_{m}-\tau\overline{x}_{m}|\\
	&=|l_{\nu+1}|\big|\psi(\tau x',\tau^{2}y_{2}\cdots,\tau^{k}y_{k})-\psi(\tau\overline{x}',\tau^{2}\overline{y}_{2}\cdots,\tau^{k}\overline{y}_{k})\big|\\
	&\leq~C_{b}\theta\big|\psi(\tau\overline{x'},\tau^{2}\overline{y_{2}}\cdots,\tau^{k}\overline{y_{k}})-\psi(\tau\underline{x'},\tau^{2}\underline{y_{2}}\cdots,\tau^{k}\underline{y_{k}}) \big|\\
	&\leq C_b\theta\big|\psi(\tau x',\tau^{2}y_{2}\cdots,\tau^{k}y_{k})-\psi(\tau\overline{x'},\tau^{2}y_{2}\cdots,\tau^{k}y_{k}) \big|\\
	&+C_b\theta\big|\psi(\tau\overline{x'},\tau^{2}y_{2}\cdots,\tau^{k}y_{k})-\psi(\tau\overline{x'},\tau^{2}\overline{y}_{2}\cdots,\tau^{k}\overline{y}_{k}) \big|.
	\end{aligned}
	\end{equation}
	In order to estimate the right hand side of inequation \eqref{sca6}, let us observe that the following holds:
	\begin{equation}\label{1f}
	\|\nabla_{x'}\psi\|_{L^{\infty}(B(s))}\leq\tilde{\delta}\omega_{\nabla\psi}(s),
	\end{equation}
	because
	\begin{equation}\label{1g}
	\|\psi\|_{C^{1,\text{Dini}}}\leq\tilde{\delta},~~~~\psi(0,0)=0\hspace{0.3cm}~\text{and}\hspace{0.3cm}\Delta_{x'}\psi(0,0)=0.
	\end{equation}
	So, in view of \eqref{1f} and Taylor's formula the first term of the rightmost extreme inequality in \eqref{sca6} can be estimated as follows:
	\begin{equation}\label{sca7}
	\begin{aligned}{}
	\big|\psi(\tau x',\tau^{2}y_{2}\cdots,\tau^{k}y_{k})-\psi(\tau\overline{x}',\tau^{2}y_{2}\cdots,\tau^{k}y_{k}) \big|&\leq\tilde{\delta}\tau|x'-\overline{x}'|\omega_{\nabla\psi}(\tau)\\ &\leq C_{2} \tilde{\delta}\tau\omega_{\nabla\psi}(\tau)d(p,\overline{p}) \leq C_{2} \tilde{\delta}\tau\omega(\tau)d(p,\overline{p})
	\end{aligned}
	\end{equation}
	where, we use $|x'-\overline{x'}|\leq C_{2} d(p,\overline{p})$ and $\omega_{\nabla\psi}(\tau)\leq\omega(\tau).$ Now, by using the mean value theorem, we can also estimate the second term of the rightmost extreme inequality in \eqref{sca6} as follows:
	\begin{equation}\label{sca8}
	\begin{aligned}{}
	\bigg|\psi(\tau\overline{x'},\tau^{2}y_{2}\cdots,\tau^{k}y_{k})-\psi(\tau\overline{x'},\tau^{2}\overline{y}_{2}\cdots,\tau^{k}\overline{y}_{k}) \bigg|\leq C_{3}\tilde{\delta}\tau^{1+\alpha}d(p,\overline{p})
	\leq C_{3}\tau\omega({\tau})\tilde{\delta}d(p,\overline{p}).
	\end{aligned}
	\end{equation}
	The first inequality follows since $\tau^{i}\leq\tau^{1+\alpha}$ for any $2\leq i$.
	In \eqref{sca8}, we have used $\tau^{\alpha}\leq \omega(\tau)$ and $\tau<1.$ Now, let us take $\tilde{C}=\max\{C_{2},~C_{3}\}$ and choose
	\begin{equation}\label{1.200}
	\tilde{\delta}=\min\Big\{\frac{\delta}{2C_{b} \tilde{C}\theta},~~\frac{\delta}{C_{b}2m^{2}\theta}\Big\},
	\end{equation}
	Therefore, by the above choice of $\tilde{\delta}$ and using the inequalities from \eqref{sca7}, \eqref{sca8} in \eqref{sca6} we get  \eqref{1c}.
	
	\noindent
	\textbf{(3-(b))~\emph{Affine approximation of the solution $u$ on the non-characteristic portion of the boundary.}}
	Now, we show that $\{L_{\nu}\}$ the sequence of polynomial converges to linear function $\mathbb{L}$ as $\nu\rightarrow\infty.$ Moreover, $\mathbb{L}$ is an affine approximation of solution to \eqref{dp1} on $e\in\partial\Omega.$ By translation, in a similar way one can show that at each point of the non-characteristic portion of the boundary, there is an affine approximation of solution to \eqref{dp1}.
	More precisely, given any non-characteristic point $p_0\in\partial\Omega$ there exists an affine function $L_{p_0}$ such that
	\begin{equation}\label{1m}
	|u(p)-L_{p_{0}}(p)|\leq C_{aff}d(p,p_0)W(d(p,p_0)).
	\end{equation}
	Moreover, $W$ can be chosen to be $\alpha-$decreasing in the sense of definition \ref{alphad}.
	Now, let us try to prove \eqref{1m} for $p_0=e\in\partial\Omega$ by assuming that all the previous step holds at $e.$ Let us take an arbitrary $p\in\partial\Omega\cap B(1)$ and choose an integer $\nu\in\mathbb{N}$ such that $\sigma^{\nu+1}\leq |p|\leq \sigma^{\nu}.$ Let us define $\mathbb{L}=\lim_{\nu\rightarrow\infty}L_\nu,$ where $L_\nu$ is from above step and consider
	\begin{equation}\label{sca9}
	\begin{aligned}{}
	&|u(p)-\mathbb{L}(p)|\leq |u(p)-L_{\nu}(p)|+|L_{\nu}(p)-\mathbb{L}(p)|\\
	&\leq\sigma^{\nu}\omega(\sigma^{\nu})+\sum^{\infty}_{j=0}|L_{\nu+j}-L_{\nu+j+1}|_{\Omega\cap B(\sigma^{\nu})}~~~~~~\leq\sigma^{\nu}\omega(\sigma^{\nu})+C_{b}\sigma^{\nu}\sum^{\infty}_{j=\nu}\omega(\sigma^{j}).
	\end{aligned}
	\end{equation}
	The last step follows from \eqref{a1} and \eqref{1b}.
	In order to estimate the sum in the last line of \eqref{sca9}, let us observe that for any fixed $j\in\mathbb{N},$ it follows from \eqref{omega} that
	\begin{equation}\label{1m1}
	\omega(\sigma^{j})\leq \frac{1}{\tilde{\delta}}\sum^{j/2}_{l=0}\omega_{1}(\sigma^{j-l})\omega_{2}(\sigma^{l})+\frac{1}{\tilde{\delta}}\sum^{j}_{l=j/2}\omega_{1}(\sigma^{j-l})\omega_{2}(\sigma^{l})+\sigma^{j\alpha}.
	\end{equation}
	Therefore, we have
	\begin{equation}\label{sca10}
	\begin{aligned}{}
	\sum^{\infty}_{\nu=j}\omega(\sigma^{j})&\leq \frac{1}{\tilde{\delta}}\sum^{\infty}_{\nu=j}\sum^{j/2}_{l=0}\omega_{1}(\sigma^{j-l})\omega_{2}(\sigma^{l})+\frac{1}{\tilde{\delta}}\sum^{\infty}_{\nu=j}\sum^{j}_{l=j/2}\omega_{1}(\sigma^{j-l})\omega_{2}(\sigma^{l})+\sum^{\infty}_{\nu=j}\sigma^{j\alpha}\\
	&\leq\frac{C}{\tilde{\delta}}\sum^{\infty}_{\nu=j}\omega_{1}(\sigma^{j/2})+\frac{1}{\tilde{\delta}}\sum^{\infty}_{\nu=j}\sum^{j}_{j=l/2}\omega_{1}(\sigma^{j-l})\omega_{2}(\sigma^{l})+\sum^{\infty}_{\nu=j}\sigma^{j\alpha}\\
	&=D+E+F.
	\end{aligned}
	\end{equation}
	In the second line we have used $\sum^{\infty}_{j=0}\omega_{2}(\sigma^{j})\leq C,$ see \eqref{m7}. In order to estimate $D, E$ and $F$ in \eqref{sca10}, let us define
	\begin{equation}\label{sca11}
	\begin{aligned}{}
	&W_{1}(\epsilon):=\sup_{a\geq0}\int^{a+\epsilon^{1/2}}_{a}\frac{\omega_{1}(s)}{s}ds,~~~~~W_{2}(\epsilon):=\epsilon^{\alpha/2},\\
	&W_{3}(\epsilon):=\sup_{a\geq0}\int^{a+\epsilon^{\frac{1}{2}}}_{a}[g^{**}(s)s^{\frac{q}{Q}}]^{\frac{1}{q}}\frac{ds}{s}~~~\text{and}~~~W_{4}(\epsilon)=\sup_{a\geq0}\int^{a+\epsilon^{1/2}}_{a}\frac{\tilde{\omega}_{2}(s)}{s}ds.
	\end{aligned}
	\end{equation}
	\noindent
	\textbf{Estimate~for~D:}~We estimate $D$ as follows:
	\begin{equation}\label{sm1}
	D\leq C\int^{\sigma^{\frac{\nu}{2}}}_{0}\frac{\omega_{1}(s)}{s}ds\leq CW_{1}(\sigma^{\nu})~~~\text{in~view~of~definition~of~}~W_{1}.
	\end{equation}
	\noindent
	\textbf{Estimate~for~E:}~We use the standard formula for geometric series to get:
	\begin{equation}\label{sm2}
	E\leq C\sigma^{\nu\alpha}=CW_{2}(\sigma^{2\nu})\leq CW_{2}(\sigma^{\nu})~~~\text{in~view~of~definition~of}~W_{2},
	\end{equation}
	where the last inequality follows because $\sigma<1$ and so $\sigma^{2\nu}<\sigma^{\nu}.$
	
	\noindent
	\textbf{Estimate~for~F:}~~Let us observe
	\begin{align}\nonumber
	F\leq&C\Big(\sum^{\infty}_{j=\frac{\nu}{2}}\omega_{2}(\sigma^{j})\Big)\Big(\sum^{\infty}_{j=1}\omega_{1}(\sigma^{j})\Big)\nonumber
	\leq C\Big(\sum^{\infty}_{j=\nu/2}\omega_{2}(\sigma^{j})\Big)~~~~~(\text{by}~\eqref{m5}) \\\nonumber
	&=\Big[C_{II}\sum^{\infty}_{j=\nu/2}\sigma^{j}\Big(\frac{1}{|\Omega\cap B(\sigma^{j})|}\int_{\Omega\cap B(\sigma^{j})}|g|^{q}\Big)^{\frac{1}{q}}+\underbrace{\sum^{\infty}_{j=\nu/2}\omega_{f}(\sigma^{j})}_{I}+\sum^{\infty}_{j=\nu/2}\sigma^{j\alpha}\Big]\\\nonumber
	\leq&\Big[\tilde{C}\underbrace{\int^{\sigma^{\frac{\nu Q}{2}}}_{0}\big[g^{**}(s)s^{\frac{q}{Q}}\big]^{\frac{1}{q}}\frac{ds}{s}}_{III}+\underbrace{\sum^{\infty}_{j=\nu/2}\tilde{\omega}_{2}(\sigma^{j})}_{II}+\sum^{\infty}_{j=\nu/2}\sigma^{j\alpha}\Big]\\
	\leq&C_{1}\underbrace{W_{3}(\sigma^{\nu})}_{IV}+C_{2}W_{4}(\sigma^{\nu})+C_{3}W_{2}(\sigma^{\nu})\label{sm3},
	\end{align}
	where we have used the fact that $\sigma^{\frac{Q\nu}{2}}\leq\sigma^{\frac{\nu}{2}}$ in deducing $II$ from  $I$ (which follows since $\sigma<1$ and $Q\geq2$) and $\omega_{f}(s)\leq \tilde{\omega}_{2}(s)$ in deducing  $IV$ from  $III$. From \eqref{sm1}, \eqref{sm2}, \eqref{sm3} and the choice of $|p|\approx\sigma^{\nu},$ we find that $D,E$ and $F\rightarrow0$ as $|p|\rightarrow0.$ It is also clear from the definition of $W_{2}$ that it is nondecreasing. Moreover, we can also assume that each $W_{i}$ is non-decreasing.
	
	Without loss of generality we can assume that $W_{j}(\cdot),$ for $j=1\cdots4,$ are $\alpha$ decreasing in the sense of definition \eqref{alphad}.
	Indeed, let us first consider the case $W_{1}.$ From the fact that $\omega_{1}(\cdot)$ is a modulus of continuity and concave, we have that $W_{1}(\cdot)$ satisfies all the properties of the definition \eqref{kartik} and hence is also a modulus of continuity. Using Theorem \ref{KA2.5}, without loss of generality, we can assume $W_{1}$ is also concave. Now, we can replace $W_{1}(s)$ by $W_{1}(s^{\alpha}),$ if necessary, we can assume $W_{1}(\cdot)$ is $\alpha$ decreasing. Since $W_{4}(\cdot)$ is same as $W_{1}$ so the assertion for $W_{4}$ also follows.
	Now let us consider the case of $W_{3}.$ From the definition \eqref{kartik}, it is clear that $W_{3}$ is a modulus of continuity. Using Theorem \ref{KA2.5}, without loss of generality, we can assume $W_{3}(\cdot)$ is also concave. Now replacing $W_{3}(s)$ by $W_{3}(s^{\alpha}),$ if necessary, we can assume that $W_{3}(\cdot)$ is $\alpha$ decreasing.	
	
	Without loss of generality, we will denote the changed $W_{i}$ with the same notion and assume that these are $\alpha$ decreasing. With the above $W_{i}(\cdot)$ in the hand we define a new $\alpha$ decreasing function $W(\cdot)$ as follows:
	\begin{equation}\label{modu}
	W(s):=W_{1}(s)+W_{2}(s)+W_{3}(s)+W_{4}(s),
	\end{equation}
	which is again $\alpha$ decreasing. So in view of $|p|\approx\sigma^{\nu},$  \eqref{sca10}, \eqref{sm1}, \eqref{sm2} and \eqref{sm3} along with \eqref{sca9} we have
	\begin{equation}
	|u(p)-\mathbb{L}(p)|\leq C\sigma^{\nu}W(\sigma^{\nu})=C |p|W(|p|),
	\end{equation}
	and this completes the proof of this step.
	\subsection{Interior~estimate}
	In the next two steps we prove the continuity of the horizontal gradient on the non-characteristic portion of the boundary and up to the boundary, respectively. In the proof of these results we need a scale invariant interior estimate, see Corollary \ref{intr1}. This estimate is a suitable adaptation of \cite[Corollary 3.2]{BG} to our set up. Since the proof follows on the same line as of the boundary case, therefore, we just sketch the proof instead of giving the complete details.
	\begin{corollary}\label{intr1}
		Given $0<\tau\leq1$, let $u\in\mathscr{L}^{1,2}_{loc}(\V_\tau) \cap C(\overline{\V_\tau})$ be a weak solution to
		\begin{equation}\label{dp01}
		\sum_{i,j=1}^m X_i^\star (a_{ij}X_j u)  = \sum_{i=1}^m X_i^\star f_i + g\ \ \ \text{in}~~~B(\tau),
		\end{equation}
		where $f= (f_1,...,f_m) \in \Tau^{0, \text{Dini}}(B(\tau)),$ $a_{ij} \in \Tau^{0, \text{Dini}}(B(\tau)),$ $a_{ij}$ satisfies \eqref{ea01} and $g \in L^{q}(B(\tau))$ with $~2q>Q$. Then, $u \in \Tau^{1}(B(\tau/2)).$ Moreover, we have the following estimates
		\begin{equation}\label{i1311}
		|\nabla_\cH u(e)|\leq\frac{C\|u\|_{L^{\infty}(B(\tau))}}{\tau}\big(1+W(\tau)\big),
		\end{equation}
		and
		\begin{equation}\label{i1331}
		|\nabla_\cH u(p)-\nabla_\cH u(e)|  \leq C \|u\|_{L^{\infty}(B(\tau)}\big(W(|p|)+\frac{|p|^\alpha}{\tau^{1+\alpha}}\big),
		\end{equation}
		$p\in B(\tau/2),$ where $C>0$ is a universal constant and $W(\cdot)$ is a given by \eqref{modu}.
	\end{corollary}
	\begin{proof}
		Given a function $u$ let us define a new function $v(p)=u(\delta_{\tau}(p))$ for $p\in B(1).$  It is clear that $v$ satisfies the following equation
		\begin{equation}
		\sum_{i,j=1}^m X_i^\star(a_{ij,\tau} X_j v) = \sum_{i=1}^m X_i^\star  f_{i,\tau} + g_{\tau}~~~~\text{in}~~~B(1),
		\end{equation}
		where $f_{i,\tau}(p)=\tau f_{i}(\delta_{\tau}(p))$ and $g_{\tau}(p)=\tau^{2}g(\delta_{\tau}(p)).$ Without loss of generality, we can assume that $\|v\|_{L^{\infty}(B(1))}\leq 1,$ since otherwise we consider the function $v(p)=\frac{u(\delta_{\tau}(p))}{\|u\|_{L^{\infty}(B(\tau))}}.$ In order to prove \eqref{i1311}, it is sufficient to prove that there exists a sequence of polynomials $\{L_{\nu}\}$ of the form $L_{\nu}(p)=a_{\nu}+\langle b_{\nu},x\rangle,$ where $(x,y_{2},\cdots,y_k)$ denote the logarithmic coordinate of $p,$ such that
		\begin{equation}\label{1.53}
		\begin{aligned}{}
		\|v-L_\nu\|_{L^{\infty}(B(\sigma^\nu))}&\leq \sigma^\nu\omega(\sigma^\nu)~~~~\text{and}~~|b_\nu|\leq C,\\
		|a_{\nu+1}-a_\nu|&\leq C\sigma^{\nu}\omega(\sigma^{\nu}),~~~~
		|b_{\nu+1}-b_\nu|\leq C\omega(\sigma^\nu).
		\end{aligned}
		\end{equation}
		As in the proof of Step (3), the above inequalities \eqref{1.53} follow by the induction argument. Here, we skip the details. Hence, using the estimates from before (adapted to the interior case), one sees that
		\begin{equation}
		|\nabla_\cH v(e)|\leq C(1+W(|p|)).
		\end{equation}
		Therefore, scaling back to $u$ we get
		\begin{equation}
		|\nabla_\cH u(e)|\leq\frac{C\|u\|_{L^{\infty}(B(\tau))}}{\tau}\big(1+W(|p|)\big).
		\end{equation}
		Analogously, we also get
		\begin{equation}\label{1.57}
		|\nabla_\cH  v(p)-\nabla_\cH v(e)|\leq C\big(\tau W(\tau|p|)+|p|^\alpha\big),
		\end{equation}
		for all $p\in B(1/2).$ Re-scaling the inequality \eqref{1.57} back to $u,$ we get the following inequality
		\begin{equation*}
		|\nabla_\cH u(\delta_\tau(p))-\nabla_\cH u(e)|\leq C\frac{\|u\|_{L^{\infty}(B(\tau)}}{r}\big(\tau W(\tau|p|)+|p|^\alpha\big),
		\end{equation*}
		that is,
		\begin{equation*}
		|\nabla_\cH u(\delta_\tau(p))-\nabla_\cH u(e)|\leq C\|u\|_{L^{\infty}(B(\tau)}\big(W(\tau|p|)+\frac{|p|^\alpha}{\tau}\big).
		\end{equation*}
		Now, putting back $q=\delta_{\tau}p$ we get
		\begin{equation*}
		|\nabla_\cH u(q)-\nabla_\cH u(e)|\leq C\|u\|_{L^{\infty}(B(\tau)}\big(W(|q|)+\frac{|q|^\alpha}{\tau^{1+\alpha}}\big),
		\end{equation*}
		which completes the proof of the Corollary.
	\end{proof}
	Having finished the interior estimate now let us move to the next step.
	
\noindent	
	\textbf{Step-(4)~Continuity of the horizontal gradient on $\mathscr S_{1/2}.$} In the step (3), we have shown that for any $ p\in \mathscr S_{1/2},$  there is a Taylor polynomial $L_p$ of $u$ at $p.$ In this step, our objective is to show that for any (non-characteristic) points $ p_1, p_2 \in \mathscr S_{1/2}$, the following estimate holds:
	\begin{equation}\label{bdcmp1}
	|\nabla_\cH L_{p_1} - \nabla_\cH L_{p_2}| \leq C \, (W(d(p_1,p_2))),
	\end{equation}
	for some universal $C,$ where $W(.)$ is a modulus function defined by \eqref{modu}.
	
	\begin{proof}{ of \eqref{bdcmp1}.}~Let $t= d(p_1,p_2)$. We consider a ``non-tangential" point $p_3 \in \V_1$ at a (pseudo) distance from $p_1$ comparable to $t$,  i.e., let $p_3$ be such that
		\begin{equation}\label{nont}
		d(p_3, p_1) \sim t,\ d(p_3, \pa \Om)  \sim t,
		\end{equation}
		where we have assumed $d(p,\pa \Om)= \underset{p' \in \pa \Om}{\inf} \ d(p,p')$. Since $\mathscr S_1$ is a non-characteristic $C^{1,\text{Dini}}$ portion of $\pa \Om$, therefore, it is possible to find such a point $p_3$. Arguing as in the proof of \cite[Theorem 7.6]{DGP}, at any scale $t$ one can find a non-tangential pseudo-ball from inside centered at $p_3$. In fact, there exists a universal $a>0$ sufficiently small (which can be seen to depend on the Lipschitz character of $\pa \Om$ near the non-characteristic portion $\mathscr S_1$) such that for some $c_0$ universal one has
		\[
		d(p,\pa \Om) \geq  c_0 t~~~\text{for all}~~~p \in B(p_3,a t).
		\]
		This allow us to apply step (3) above and conclude that there exists a universal $C>0$ such that for all $p\in B(p_3, a t)$ we have:
		\begin{equation}\label{3.41}
		|u(p)- L_{p_1}(p)| \leq C \, t \, W(t),\ \hspace{0.5cm}\ \ |u(p) - L_{p_2}(p)|\leq C tW(t).
		\end{equation}
		Now, for $\ell=1, 2$  we note that  $v_\ell= u- L_{p_\ell}$ solves
		\begin{equation}\label{new21}
		\sum_{i,j=1}^m X_{i}^\star( a_{ij} X_j v_\ell) = \sum_{i=1}^m X_{i}^\star  F_i^\ell + g,
		\end{equation}
		where we have let
		\[
		F_i^\ell \overset{def}{=}  f_i - \sum_{j=1}^m  a_{ij} X_j L_{ p_\ell}.
		\]
		Since $f_i$ and $a_{ij}$ are Dini continuous, therefore, without loss of generality we can assume that $F^\ell_i,$ are Dini continuous. Also, from \eqref{3.41} we see that $v_\ell$ satisfies
		\begin{equation}\label{n1001}
		||v_\ell||_{L^{\infty}(B(p_3,a t))} \leq C t W(t), \ \ \ \ \ell=1, 2.
		\end{equation}
		With \eqref{n1001} in hand, we can now use the interior estimate \eqref{i1311} in Corollary \ref{intr1} in the pseudo-ball $B(p_3,a t)$ to obtain the following estimate for $\ell = 1, 2$
		\begin{eqnarray} \label{1.28}
		|\nabla_\cH v(p)|=|\nabla_\cH u(p)-\nabla_\cH L_{p_\ell}(p)| & \leq & \frac{C}{t} ||u-L_{p_\ell}||_{L^{\infty}(B(p_0,t))}(1 + W(t))\nonumber\\
		&& \leq C \, W(t),
		\end{eqnarray}
		by \eqref{3.41}.
		From \eqref{1.28} and the triangle inequality we obtain that the following estimate holds:
		\begin{equation*}
		|\nabla_\cH L_{p_1} - \nabla_\cH L_{p_2}| \leq C\, W(t) \leq  C \, (W(d(p_1,p_2))),
		\end{equation*}
		where we have used $t\sim d(p_1,p_2),$ which is the desired estimate \eqref{bdcmp1}.
	\end{proof}
	
	\noindent
	\textbf{Step-(5)~Patching the interior and boundary estimate:~}
	In this step we prove that the horizontal gradient of a weak solution to \eqref{dp1} is $\Gamma^{1}$ up to the boundary. First, we observe that there is an $\ve>0$ sufficiently small such that for any $p\in \V_\ve,$ there exists $ p_0 \in \mathscr S_{1/2}$ such that
	\begin{equation}\label{eqt1}
	d(p, p_0) = d(p,\pa \Om).
	\end{equation}
	To finish the proof of the Theorem \ref{main}, we will show that for all $p,p^{\star}\in \V_\ve$ we have:
	\begin{equation}\label{fop1}
	|\nabla_\cH u(p)- \nabla_\cH u(p^{\star})|\leq  C^\star \left(W(d(p,p^{\star}))\right),
	\end{equation}
	for some  universal constant $C^\star >0$. Let $p, p^{\star} \in \V_\ve$ be the two given points. Let $p_0,  p_0^{\star}$ be the corresponding points in $\mathscr S_{1/2}$ for which \eqref{eqt1} holds. Let us write $\delta(p)=d(p,\pa \Om)$ for  $p \in \Om$. Without loss of generality we may assume that
	\begin{eqnarray} \label{1.33}
	\delta(p)= \min \{\delta(p), \delta(p^{\star})\}.
	\end{eqnarray}
	By step-(3), there exists a first-order polynomial $L_{p_0}$ such that for every $q \in \V_1$ we have
	\begin{equation}\label{fopb1}
	|u(q) - L_{p_0}(q)|\leq C_2d(p_0,q) W(d(p_0,q)),
	\end{equation}
	where $ p_0$ is as in \eqref{eqt1}.
	Now, there are two possibilities:
	\begin{itemize}
		\item[(a)] $d(p, p^{\star}) \leq \frac{\delta(p)}{2}$;
		\item[(b)] $d(p, p^{\star}) > \frac{\delta(p)}{2}$.
	\end{itemize}
	\medskip
	\noindent \textbf{(a)}~In view of \eqref{1.33}, it is clear that $B(p, \delta(p)) \subset \Omega.$ Now, let us consider the function $v:=u-L_{p_0},$ where $p_0 \in\mathscr S_{1/2}$ is the point corresponding to $p$ discussed above and $L_{p_0}$ is the polynomial from step-(3). Again it is easy to see that $v$ satisfies an equation of the type \eqref{new21} in $B(p, \delta(p)) \subset \Omega.$ Now, we can apply Corollary \ref{intr1}~ (interior estimate) along with \eqref{fopb1} to get the following estimate:
	\begin{equation}\label{sup11}
	||v||_{L^{\infty}(B(p, \delta(p))} \leq \tilde{C}_2 \delta(p) W(\delta(p)),
	\end{equation}
	for some $\tilde{C}_2>0.$ Since $p^{\star} \in B(p, \delta(p)/2)$, so by using the interior estimate \eqref{i1331} (Corollary \ref{intr1}) and \eqref{sup11}, we find that for some $\tilde C$ depending also on $\tilde{C}_2$ the following estimates hold:
	\begin{align}\label{sup21}
	& |\nabla_\cH v(p) - \nabla_\cH v(p^{\star})|=|\nabla_\cH u(p) - \nabla_\cH u(p^{\star})|
	\\
	& \leq C \left(W(d(p,p^{\star})) \, [||u-L_{p_0}||_{L^{\infty}(B(p,\delta(p)))}] +\frac{|d(p,p^{\star})|^{\alpha}}{\delta(p)^{1+\alpha}}[||u-L_{ p_0}||_{L^{\infty}(B(p,\delta(p)))}]\right) \notag\\
	&\leq C \left(W(d(p,p^{\star})) \,[ \delta(p)\,W(\delta(p))] +\frac{|d(p,p^{\star})|^{\alpha}}{\delta(p)^{\alpha}}[ W(\delta(p))] \right).
	\notag
	\end{align}
	Now, $\alpha-$decreasing property of  $W(\cdot)$ implies
	\begin{equation}\label{1.331}
	\begin{aligned}{}
	&\frac{|d(p,p^{\star})|^{\alpha}}{\delta(p)^{\alpha}}[ W(\delta(p))]\leq W(d(p,p^{\star})).
	\end{aligned}
	\end{equation}
	With the help of \eqref{1.331}, \eqref{sup21} can be rewritten as follows:
	$$|\nabla_\cH u(p) - \nabla_\cH u(p^{\star})|\leq C(W(d(p,p^{\star}))),$$
	which gives \eqref{fop1}.
	
	\noindent \textbf{(b)}~In this case, we have $d(p, p^{\star})>\frac{\delta(p)}{2}$ and from \eqref{eqt1} we get
	\begin{eqnarray} \label{1.36}
	d(p,p_0)=d(p,\pa\Om)=\delta(p)<2 d(p, p^{\star}).
	\end{eqnarray}
	Let us recall the following pseudo-triangle inequality for $d$
	\begin{equation}\label{t11}
	d(p,p') \leq C_0 (d(p,p'') + d(p'',p')),
	\end{equation}
	for all $p, p', p''\in \G$, and a universal $C_0>0$.
	From \eqref{1.36} and \eqref{t11} we get
	\begin{equation}\label{n1011}
	d(p^{\star}, p_0) \leq C_0(d(p^{\star},p) + d(p,p_0)) \leq  C_0 (d(p^{\star},p) + 2 d(p^{\star},p)) = 3 C_0 d(p, p^{\star}).
	\end{equation}
	Since, we also have $d(p^{\star},p_0)\geq d(p^{\star},\pa \Om) = \delta(p^{\star}),$ therefore, in view of \eqref{n1011}, we get
	\begin{equation}\label{n1031}
	\delta(p^{\star})\leq 3C_0 d(p,p^{\star}).
	\end{equation}
	So by combining \eqref{t11}, \eqref{n1011} and \eqref{n1031} we finally obtain
	\begin{equation}\label{dist1}
	d(p_0, p_0^{\star}) \leq C_0 (d(p_0, p^{\star}) + d(p^{\star}, p_0^{\star})) = C_0 (d(p_0,p^{\star}) + \delta(p^{\star})) \leq 6 C_0^2 d(p, p^{\star}).
	\end{equation}
	Let $b$ be the universal constant in the existence of a non-tangential (pseudo)-ball in the previous step-(4). Therefore, from Step -(3), we have the following estimates:
	\begin{equation}\label{apt1}
	||u - L_{ p_0}||_{L^{\infty}(B(p, b\delta(p))} \leq  \tilde K_0 \delta(p) W(\delta(p)),\ \ \ \ \ ||u - L_{p_0^{\star}}||_{L^{\infty}(B(p, b\delta(p^{\star}))} \leq  \tilde K_0 \delta(p^{\star}) W(\delta(p^{\star})).
	\end{equation}
	Let us define $v= u - L_{p_0},$ and observe that $v$ satisfies an equation of the type \eqref{new21}. Therefore, arguing as in \eqref{3.41}-\eqref{1.28} and using the former estimate \eqref{apt1} in $B(p,b \delta(p))$ along with the interior estimate in Corollary \ref{intr1}, we obtain that for some universal constant $C>0,$ we have
	\begin{equation}\label{imo21}
	|\nabla_\cH  u(p) - \nabla_\cH L_{p_0}|= |\nabla_\cH v(p)|  \leq C  \, W(\delta(p)) \leq C \, W (d(p, p^{\star})),
	\end{equation}
	where in the last inequality we have used $\delta(p) \leq 2 d(p, p^{\star}).$ Arguing as before \eqref{imo21}, we obtain
	\begin{equation}\label{imo11}
	|\nabla_\cH  u(p^{\star})-\nabla_\cH L_{p_0^{\star}}|\leq C W(\delta(p^{\star}))\leq C \, W(d(p, p^{\star}))
	\end{equation}
	by \eqref{n1031}.
	Now, from \eqref{bdcmp1} and  \eqref{dist1} we have
	\begin{eqnarray}\label{finali1}
	|\nabla_\cH L_{p_0} - \nabla_\cH L_{p_0^{\star}}|
	&&\leq C \, W(d(p_0, p_0^{\star})) \le C \, W(d(p, p^{\star})).
	\end{eqnarray}
	Applying the triangle inequality along with the estimates \eqref{imo21}, \eqref{imo11} and \eqref{finali1} we get
	$$
	|\nabla_\cH u(p)- \nabla_\cH u(p^{\star})|\leq  C^\star \left(W(d(p,p^{\star}))\right).
	$$ This completes the proof of the Theorem \ref{main}.
\end{proof}

\section{Acknowledgement}
The authors are thankful to Prof. Agnid Banerjee for suggesting the problem and for various stimulating discussions on this topic.

\end{document}